\documentclass{article} % Anonymized submission

\usepackage{amsmath,amssymb,amsthm}
\usepackage{bm}
\usepackage{hyperref} 
\usepackage{comment}
\usepackage{microtype}
\usepackage{graphicx}
\usepackage{subfigure}
\usepackage{booktabs}

\usepackage{multirow}

\usepackage[accepted]{icml2021}

\icmltitlerunning{Non-Convex Exact Recovery via Projected Power Method}

\newcommand{\be}{\begin{equation}}
\newcommand{\ee}{\end{equation}}
\newcommand{\st}{\mbox{s.t.}}
\newcommand{\argmin}{\mathop{\mathrm{arg\,min}}}
\newcommand{\argmax}{\mathop{\mathrm{arg\,max}}}
\newcommand{\ve}{\mathrm{vec}}

\newtheorem{lemma}{Lemma}
\newtheorem{thm}{Theorem}
\newtheorem{coro}{Corollary}
\newtheorem{defi}{Definition}
\newtheorem{prop}{Proposition}

\def\bA{\bm{A}}
\def\bB{\bm{B}}
\def\bC{\bm{C}}
\def\bCp{\bm{C}^\prime}
\def\bE{\bm{E}}
\def\bG{\bm{G}}
\def\bH{\bm{H}}
\def\bHp{\bm{H}^\prime}
\def\bHs{\bm{H}^*}
\def\bI{\bm{I}}
\def\bLam{\bm{\Lambda}}

\def\bQ{\bm{Q}}
\def\bU{\bm{U}}
\def\bV{\bm{V}}
\def\bVp{\bm{V}^\prime}
\def\bZ{\bm{Z}}
\def\b0{\bm{0}}
\def\ba{\bm{a}}

\def\bc{\bm{c}}

\def\be{\bm{e}}
\def\bh{\bm{h}}
\def\bo{\bm{1}}
\def\bpi{\bm{\pi}}
\def\bu{\bm{u}}

\def\bw{\bm{w}}
\def\bx{\bm{x}}

\def\mH{\mathcal{H}}

\def\mI{\mathcal{I}}
\def\mJ{\mathcal{J}}
\def\mO{\mathcal{O}}
\def\mP{\mathcal{P}}
\def\mS{\mathcal{S}}
\def\mT{\mathcal{T}}
\def\bb{\bm{b}}
\def\E{\mathbb{E}}
\def\H{\mathbb{H}}
\def\M{\mathbb{M}}
\def\P{\mathbb{P}}
\def\R{\mathbb{R}}

\begin{document}

\twocolumn[
\icmltitle{Optimal Non-Convex Exact Recovery in Stochastic Block Model \\ via Projected Power Method}


\begin{icmlauthorlist}
	\icmlauthor{Peng Wang}{to}
	\icmlauthor{Huikang Liu}{goo}
	\icmlauthor{Zirui Zhou}{ed}
	\icmlauthor{Anthony Man-Cho So}{to}
\end{icmlauthorlist}

\icmlaffiliation{to}{Department of Systems Engineering and Engineering Management, The Chinese University of Hong Kong, Hong Kong}
\icmlaffiliation{goo}{Business School, Imperial College London, London, United Kingdom}
\icmlaffiliation{ed}{Huawei Technologies Canada Co., Ltd., Burnaby, Canada}

\icmlcorrespondingauthor{Huikang Liu}{hkliu2014@gmail.com}

\icmlkeywords{Machine Learning, ICML}

\vskip 0.2in
]



\printAffiliationsAndNotice{}  % leave blank if no need to mention equal contribution

\begin{abstract}
In this paper, we study the problem of exact community recovery in the symmetric stochastic block model, where a graph of $n$ vertices is randomly generated by partitioning the vertices into $K \ge 2$ equal-sized communities and then connecting each pair of vertices with probability that depends on their community memberships. Although the maximum-likelihood formulation of this problem is discrete and non-convex, we propose to tackle it directly using projected power iterations with an initialization that satisfies a partial recovery condition. Such an initialization can be obtained by a host of existing methods. We show that in the logarithmic degree regime of the considered problem, the proposed method can exactly recover the underlying communities at the information-theoretic limit. Moreover, with a qualified initialization, it runs in $\mO(n\log^2n/\log\log n)$ time, which is competitive with existing state-of-the-art methods. We also present numerical results of the proposed method to support and complement our theoretical development.  % We show that in the logarithmic degree regime of the considered problem, the proposed method exactly recovers the underlying communities nearly at the information-theoretic limit up to a factor of $\sqrt{2}$ in $\mO(n\log^2n/\log\log n)$ time with high probability, which is competitive with existing state-of-the-art methods. We also present numerical results of the proposed method to support and complement our theoretical development.  
\end{abstract}

\section{Introduction}\label{sec:intro}

Community detection is a fundamental task in network analysis and has found wide applications in diverse fields, such as social science \citep{girvan2002community}, physics \citep{newman2004finding}, and machine learning \citep{shi2000normalized}, just to name a few. As the study of community detection grows, a large variety of theories and algorithms have been proposed in the past decades for addressing different tasks under different settings. To better validate and compare these theories and algorithms, the stochastic block model (SBM), which tends to generate graphs containing underlying community structures, is widely used as a canonical model for studying community detection. In particular, substantial advances have been made in recent years on understanding the fundamental limits of community detection and developing algorithms for tackling different recovery tasks in the SBM; see, e.g., \citet{abbe2017community} and the references therein. 

In this work, we consider the problem of exactly recovering the communities in the symmetric SBM. Specifically, given $n$ nodes that are partitioned into $K \ge 2$ unknown communities of equal size, a random graph is generated by independently connecting each pair of vertices with probability $p$ if they belong to the same community and with probability $q$ otherwise. The goal is to recover the underlying communities exactly by only observing one realization of the graph. In the logarithmic degree regime of the considered SBM, i.e., $p=\alpha\log n/n$ and $q=\beta\log n/n$ for some $\alpha,\beta>0$, this problem exhibits a sharp information-theoretic threshold: it is possible to achieve exact recovery if $\sqrt{\alpha}-\sqrt{\beta} > \sqrt{K}$ and is impossible if $\sqrt{\alpha}-\sqrt{\beta} < \sqrt{K}$ \citep{abbe2015community}. Then, it is of interest to design computationally tractable methods that can achieve exact recovery under a condition on $\alpha$ and $\beta$ that meets the information-theoretic limit. In the past years, many algorithms have been proposed to achieve this task, such as spectral clustering \citep{mcsherry2001spectral,su2019strong,yun2014accurate, yun2016optimal}, SDP-based approach \cite{amini2018semidefinite,fei2018exponential,fei2020achieving,li2018convex}, and likelihood-based approach \cite{amini2013pseudo,gao2017achieving,zhang2016minimax,zhou2020rate}. However, most of these algorithms have a time complexity that is at least quadratic in $n$, which usually does not scale well to large-scale problems. 

% In this work, we consider the problem of exactly recovering the communities in the symmetric SBM. Specifically, given $n$ nodes that are partitioned into $K$ unknown communities of equal size, a random graph is generated by independently connecting each pair of vertices with probability $p$ if they belong to the same community and with probability $q$ otherwise. The goal is to recover the underlying communities exactly by only observing one realization of the graph. In the logarithmic degree regime of the considered SBM, i.e., $p=\alpha\log n/n$ and $q=\beta\log n/n$ for some $\alpha,\beta>0$, this problem exhibits a sharp information-theoretic threshold: it is possible to achieve exact recovery if $\sqrt{\alpha}-\sqrt{\beta} > \sqrt{K}$ and is impossible if $\sqrt{\alpha}-\sqrt{\beta} < \sqrt{K}$ \citep{abbe2015community}.  In view of this threshold, one question is to design a method that achieves exact recovery as nearly as possible down to the threshold.  In the past years, many algorithms have been proposed to achieve this task, such as spectral clustering \citep{mcsherry2001spectral,su2019strong,yun2014accurate}, SDP-based approach \cite{amini2018semidefinite,fei2018exponential,li2018convex}, and likelihood-based approach \cite{gao2017achieving}. However, these algorithms in general have a time complexity that is polynomial in $n$, which usually does not scale well to large-scale problems. 

In the symmetric SBM, the maximum likelihood (ML) estimation problem is formulated as
\begin{align}\label{MLE}
\max\ \left\{ \langle \bA\bH, \bH \rangle: \bH \in \mH \right\}, \tag{\textsf{MLE}}
\end{align}
where $\bA$ is the adjacency matrix of the observed graph,
\begin{align}\label{set-H}
\mH=\left\{\bH \in \R^{n\times K}:\ \bH\bo_K = \bo_n,\ \bH^T\bo_n = m\bo_K,  \right. \\ % \qquad\qquad\qquad\quad
\left.  \bH \in \{0,1\}^{n\times K}\right\} \ \quad\qquad\qquad \notag
% \mH = \{\bH \in \{0,1\}^{n\times K}:\bH\bo_K = \bo_n,\bH^T\bo_n = m\bo_K \}
\end{align}
is the discrete feasible set, $\bo_n$ is the all-one vector of dimension $n$, and $m=n/K$. It is known that an ML estimator achieves exact recovery at the information-theoretic limit, but solving Problem \eqref{MLE} is NP-hard in the worst-case. Recently, in independent lines of research, many non-convex formulations that arise in a variety of applications have been shown to be solvable, in the sense of average-case performance, by simple and scalable iterative methods. This includes phase retrieval \citep{bendory2017non, chen2019gradient}, group synchronization \citep{ling2020improved,liu2017estimation,liu2020unified,zhong2018near}, low-rank matrix recovery \citep{chi2019nonconvex}, and two-block community detection \cite{wang2020nearly}. It then naturally motivates the question of whether one can apply a similar simple and scalable method to the discrete optimization problem \eqref{MLE}. In this work, we answer this question in the affirmative by showing that a projected power method provably works for solving Problem \eqref{MLE}. As a consequence, we obtain a simple and scalable method that achieves exact recovery under the optimal condition on $\alpha$ and $\beta$.

% In view of the computational advantages and the great successes of using these methods, another question is to design a simple and efficient method for solving Problm \eqref{MLE}.  With the above observations in mind, our goal is to propose a simple and efficient algorithm that achieves exact recovery as nearly as possible down to the information-theoretic limit in the symmetric SBM. 

\subsection{Related Works}

% In this subsection, we review existing methods for exact recovery in the SBM. % and recent progress on provable non-convex approaches.

% In the SBM, some typical tasks of community detection have been extensively studied in the literature, such as exact recovery, almost exact recovery, and partial recovery. 
% We review definitions of these recovery tasks, the fundamental limit for exact recovery in the symmetric SBM, and various algorithms for exact recovery in the literature. Besides, we review some works that study simple and efficient methods for non-convex formulations.
% \subsubsection{Existing methods for exact recovery}

In the context of the SBM, \emph{exact recovery}, also named \emph{strong consistency}, requires all the communities to be identified correctly up to a permutation of labels. More precisely, exact recovery is achieved if there exists an algorithm that takes one realization of the graph as input and outputs the true partition with high probability. In the logarithmic degree regime of the binary symmetric SBM, i.e., the symmetric SBM with $K=2$, \citet{abbe2016exact} and \citet{mossel2014consistency} independently showed that it is possible to achieve exact recovery if $\sqrt{\alpha}-\sqrt{\beta} > \sqrt{2}$ and is not possible if $\sqrt{\alpha}-\sqrt{\beta} < \sqrt{2}$, thereby establishing the information-theoretic limit for exact recovery. Later, \citet{abbe2015community} generalized this result to the case of $K \ge 2$ and showed that the information-theoretic limit is $\sqrt{\alpha}-\sqrt{\beta} > \sqrt{K}$. \emph{Almost exact recovery}, also named \emph{weak consistency}, requires the recovery of all but a vanishing fraction of vertices. In \emph{partial recovery}, only a constant fraction of vertices needs to be identified correctly. It is obvious that the requirement of partial recovery is much milder than that of almost exact recovery. We refer the reader to \citet{abbe2017community} for the formal definitions of these recovery tasks and more results on the corresponding fundamental limits in the SBM. 

Over the past years, many algorithms have been proposed to tackle the problem of exact recovery in the symmetric SBM. One popular approach is spectral clustering. For example, \citet{mcsherry2001spectral} proposed a spectral partition method, which first randomly partitions the vertex set into two parts, then calls the combinatorial projection subroutine, and finally clusters the vertices by distances on the projected points. They showed that in the symmetric SBM, the proposed method achieves exact recovery if $(p-q)/\sqrt{p}\gtrsim \sqrt{\log n/n}$ and $np \gtrsim \log^6n$. Later, \citet{yun2014accurate, yun2016optimal} also presented a spectral partition method, which proceeds by applying spectral decomposition to a trimmed adjacency matrix for generating an initial partition, followed by an additional procedure for local improvement. In the considered SBM, this method achieves exact recovery down to the information-theoretic threshold in $\mO(n \mathbf{poly}\log n)$. % provided that the initial partition satisfies the requirement of almost exact recovery. 
Recently, \citet{su2019strong} showed that the standard spectral clustering, which first computes the leading $K$ eigenvectors of the graph Laplacian matrix and then applies the {\em k-means} algorithm to do clustering, achieves exact recovery under some weak conditions. These conditions can be simplified as $\sqrt{\alpha}-\sqrt{\beta} \ge c > \sqrt{K}$ for some positive constant $c$ in the symmetric SBM. In general, these spectral clustering methods run in polynomial time.  Another popular approach is convex relaxation of the ML estimation problem. In the setting of $K=2$, \citet{bandeira2018random} and \citet{hajek2016achieving,hajek2016achieving_ex} respectively showed that semidefinite programming (SDP) relaxation of the ML formulation of the binary symmetric SBM achieves exact recovery at the information-theoretic limit. In the setting of $K \ge 2$, \citet{guedon2016community} proposed a SDP relaxation of Problem \eqref{MLE} and showed a recovery error bound, which decays polynomially in the signal-to-noise ratio, for the solution to their considered SDP. Such an error bound only implies that their proposed SDP achieves almost exact recovery in our considered SBM. Following this work, \citet{fei2018exponential, fei2020achieving} proposed a new SDP relaxation of Problem \eqref{MLE} and established a more refined recovery error bound, which decays exponentially in the signal-to-noise ratio. This error bound implies that their proposed SDP achieves exact recovery provided that $(\alpha-\beta)^2\ge c\left(\alpha+(K-1)\beta\right)$ for a positive constant $c$. Besides, \citet{amini2018semidefinite} proposed another SDP relaxation of Problem \eqref{MLE} and showed that this SDP exactly recovers the communities with high probability if $(\alpha-\beta)^2 \ge cK(\alpha+K\beta)$ for a positive constant $c$. Despite the nice property of SDP-based approaches that they do not require any initial estimate of the partition or local refinement, solving the SDP problem is usually computationally prohibitive for large-scale data sets.  We refer the reader to a survey by \citet{li2018convex} for more results on convex relaxation methods for community detection.   

\vskip -0.1in
\begin{table}[t]
\caption{Comparison of recovery conditions and time complexities of the surveyed methods for exact recovery in the SBM ($K\ge 2$).}
\label{table-0}
%\vskip 0.15in
\begin{center}
\begin{footnotesize}
\begin{tabular}{ccccc}
\toprule
{\bf References} & {\bf Conditions} &  {\bf Complexities} \\
\midrule
\citet{mcsherry2001spectral} & Not optimal & Polynomial \\ 
\citet{yun2016optimal}  & Optimal & $\mO(n\mathbf{poly}\log n)$  \\
\citet{su2019strong} & Not optimal & Polynomial \\
\midrule
\citet{amini2018semidefinite} & Not optimal & Polynomial\\
\citet{fei2018exponential} & Not optimal & Polynomial \\
\midrule
\citet{abbe2015community} & Optimal & $\mO(n^{1+1/\log\log n})$ \\
\citet{gao2017achieving} & Optimal & Polynomial \\ 
\midrule
{\bf Ours} & {\bf Optimal} &  $\mO\left(\frac{n\log^2n}{\log\log n}\right)$\\
\bottomrule
\end{tabular}
\end{footnotesize}
\end{center}
\vskip -0.2in
\end{table}

We would also like to mention some algorithms for the considered problem that use other techniques. \citet{abbe2015community} developed a two-stage algorithm that consists of the Sphere-comparison sub-routine for detecting communities almost exactly and the Degree-profiling sub-routine for identifying the communities exactly. Moreover, it recovers the communities exactly with high probability all the way down to information-theoretic threshold in $\mO(n^{1+1/\log\log n})$ time in the considered SBM. Besides, \citet{gao2017achieving} proposed a two-stage algorithm that needs a weakly consistent initialization and refines it by optimizing the local penalized maximum likelihood function for each node separately. In the considered SBM, their proposed method achieves exact recovery at the information-theoretic limit in polynomial time. We refer the reader to \citet{amini2013pseudo,zhang2016minimax,zhou2020rate} for more likelihood-based approach. Recently, \citet{wang2020nearly} proposed a non-convex approach that involves initializing a generalized power method with a power method for solving a regularized ML formulation of the binary symmetric SBM. Their method runs in nearly-linear time and is among the most efficient in the literature that achieves exact recovery at the information-theoretic limit. There are still many other interesting methods for the considered problem, such as the mean field method in \citet{zhang2020theoretical}, a variant of Lloyd’s algorithm in \citet{lu2016statistical}, and the modularity-based method in \citet{cohen2020power}. Due to the limitation of space, we shall not discuss further here.

\subsection{Our Contribution}\label{sec:oc}

In this work, we propose a simple and scalable method that can achieve the optimal exact recovery threshold in the symmetric SBM. Our strategy is simply to apply the projected power method to tackle Problem \eqref{MLE} directly. Specifically, it starts with an initial point that satisfies a certain partial recovery condition and then applies projected power iterations to refine the iterates successively. In the logarithmic degree regime of the symmetric SBM, we prove that the proposed method achieves exact recovery at the information-theoretic limit. Moreover, we show that it takes $\mO(\log n/\log\log n)$ projected power iterations to obtain the underlying communities. Besides, we demonstrate that each projected power iteration is equivalent to a minimum-cost assignment problem (MCAP), which can be solved in $\mO(n\log n)$ time. These yield that the proposed method runs in $\mO\left(n\log^2n/\log\log n\right)$ time with a qualified initialization. This is competitive with the most efficient algorithms in the literature for the considered problem. It is worth noting that despite the simplicity of the proposed method, it only requires a partial recovery condition for the initial point, which is generally milder than almost exact recovery conditions that are needed for most existing two-stage algorithms; see, e.g., \citet[Algorithm 2]{gao2017achieving}, \citet[Algorithm 2]{yun2014accurate}, and \citet[Sphere-comparison algorithm]{abbe2015community}.

Our work also contributes to the emerging area of provable non-convex methods. In particular, our result indicates that 
% In recent years, the study of provable methods that directly tackles non-convex formulations has been increasingly popular.  Our work contributes to this emerging area by showing that 
the ML formulation of the symmetric SBM, albeit non-convex and discrete, can be solved via a carefully designed, yet simple, iterative procedure. Prior to our work, such discrete optimization problem is usually handled either by SDP relaxation (see, e.g., \citet{amini2018semidefinite,fei2018exponential}) or by non-convex but continuous relaxation (see, e.g., \citet{bandeira2016low}). We believe that the proposed non-convex approach can be extended to other structured discrete optimization problems; cf. \citet{liu2017discrete}. 

% We also conduct some numerical experiments on synthetic and real data sets to compare the performance of our method with some state-of-the-art ones. Numerical results show the efficacy of our method and complement our theoretical results.  
The rest of this paper is organized as follows. In Section \ref{sec:preli}, we introduce the proposed method for exact community recovery and present the main results of this paper. In Sections \ref{sec:pf-main}, we prove the main results. We then report some numerical results in Section \ref{sec:num} and conclude in Section \ref{sec:con}.

\emph{Notation.} Let $\R^n$ be the $n$-dimensional Euclidean space and $\|\cdot\|$ be the Euclidean norm. We write matrices in capital bold letters like $\bA$, vectors in bold lower case like $\bm{a}$, and scalars as plain letters. Given a matrix $\bA$, we use $\|\bA\|$ to denote its spectral norm, $\|\bA\|_F$ its Frobenius norm, and $a_{ij}$ its $(i,j)$-th element. Given a positive integer $n$, we denote by $[n]$ the set $\{1,\dots,n\}$. Given a discrete set $S$, we denote by $|S|$ the cardinality of $S$. We use $\bo_n$ and $\bE_n$ to denote the $n$-dimensional all-one vector and $n\times n$ all-one matrix, respectively. We use $\Pi_K$ to denote the collections of all $K\times K$ permutation matrices. We use $\mathbf{Bern}(p)$ to denote the Bernoulli random variable with mean $p$.

%\begin{table}[t]
%\caption{Recovery bounds, initial conditions, and time complexities of the second stage of different two-stage methods}
%\label{table-1}
%\vskip 0.15in
%\begin{center}
%\begin{small}
%\begin{tabular}{lcccr}
%\toprule
%Methods &  initial condition & time complexity  & recovery bound \\
%\midrule
%\citet{yun2014accurate} & almost exact recovery & polynomial in $n$ &  $\sqrt{\alpha}-\sqrt{\beta} > \sqrt{K}$ \\
%  & 25 & 11383 &  1064  \\
%   & 25 & 11383 &  1064  \\
%\bottomrule
%\end{tabular}
%\end{small}
%\end{center}
%\vskip -0.2in
%\end{table}

\section{Preliminaries and Main Results}\label{sec:preli}

In this section, we formally set up the considered problem in the SBM, present the proposed algorithm, and give a summary of our main results. To proceed, we introduce clustering matrices for representing community structures and the symmetric stochastic block model (SBM) for generating observed graphs. 

\begin{defi}\label{memship-matrix}
We say that $\bH \in \R^{n\times K}$ is a clustering matrix if it takes the form of 
\begin{align}\label{form-H}
h_{ik} = 
\begin{cases}
1,\quad \text{if}\ i \in \mI_k, \\
0,\quad \text{otherwise}
\end{cases}
\end{align}
\vskip -0.1in
for some $\mI_1,\dots,\mI_K$ such that ${\cup}_{k=1}^K\mI_k = [n]$ and $\mI_k \cap \mI_\ell = \emptyset$ for all $1 \le k \neq \ell \le K$. Moreover, we say that $\bH \in \R^{n\times K}$ is a balanced clustering matrix if it satisfies the above requirement with $|\mI_k|=m$ for all $k\in [K]$. For simplicity, we use $\M_{n,K},\H_{n,K}$ to denote the collections of all such clustering and balanced clustering matrices, respectively.  
\end{defi}
Intuitively, a family of sets $\mI_1,\dots,\mI_K$ represents a partition of $n$ nodes into $K$ communities such that $h_{ik}=1$ if node $i$ belongs to the community encoded by $\mI_k$ and $h_{ik}=0$ otherwise. Given a fixed $\bH \in \M_{n,K}$, $\bH\bQ$ for any $\bQ \in \Pi_K$ represents the same community structure as $\bH$ up to a permutation of the labels. 

\begin{defi}[Symmetric SBM]\label{SBM}
Let $n \ge 2$ be the number of vertices, $K \ge 2$ be the number of communities, and $p,q\in [0,1]$ be parameters of the connectivity probabilities. Furthermore, let $\bHs \in \H_{n,K}$ represent a unknown partition of $n$ vertices into $K$  equal-sized communities. We say that a random graph $G$ is generated according to the symmetric SBM with parameters $(n,K,p,q)$ and $\bHs$ if $G$ has a vertex set $V=[n]$ and the elements $\{a_{ij}\}_{1\le i \le j \le n}$ of its adjacency matrix $\bA$ are generated independently by 
\begin{align}
a_{ij} \sim \left\{
	\begin{aligned}
	\mathbf{Bern}(p),\quad & \text{if}\ \ \bh_i^{*^T}\bh^*_j = 1,\\
	\mathbf{Bern}(q),\quad & \text{if}\ \ \bh_i^{*^T}\bh^*_j = 0,
	\end{aligned}%\tag{SBM}
	\right. \label{eq:SBM}
\end{align}
where $\bh_i^{*^T}$ is the $i$-th row of $\bHs$. 
\end{defi}
Intuitively, this model states that given a true partition of $n$ vertices into $K$ unknown communities of equal size, a random graph $G$ is generated by independently connecting each pair of vertices with probability $p$ if they belong to the same community and with probability $q$ otherwise.   

Given one observation of such $G$, our goal is to develop a simple and scalable algorithm that outputs the true partition, i.e., $\bHs\bQ$ for some $\bQ \in \Pi_K$, with high probability. Since exact recovery requires the node degree to be at least logarithmic (see, e.g., \citet[Section 2.5]{abbe2017community}), we focus on the logarithmic sparsity regime of the symmetric SBM in this work, i.e.,
\begin{align}\label{p-q}
p=\alpha\frac{\log n}{n}\quad \text{and}\quad q=\beta\frac{\log n}{n},
\end{align}
where $\alpha, \beta$ are positive constants. 

The main ingredient in our approach is to apply the projected power method for solving Problem \eqref{MLE}. Specifically, the projected power step takes the form of 
\begin{align}\label{update-H}
\bH^{k+1} \in \mT(\bA\bH^k),\ \text{for all}\ k\ge 1,
\end{align}
where $\mT:\R^{n\times K} \rightrightarrows \R^{n\times K}$ denotes the projection operator onto $\mH$; i.e., for any $\bC \in \R^{n\times K}$, 
\begin{align}\label{project-H}
\mT(\bC) = \argmin\left\{ \|\bH-\bC\|_F:\ \bH \in \mH \right\}.
\end{align}
Note that Problem \eqref{MLE} can be interpreted as a principal component analysis (PCA) problem with some structural constraints. This motivates us to propose a variant of the power iteration as in \eqref{update-H} for solving it. Actually, many algorithms of similar flavor for solving PCA problems with other structural constraints have appeared in the literature; see, e.g., \citet{boumal2016nonconvex,chen2018projected,deshpande2014cone,journee2010generalized}.

One important step towards guaranteeing rapid convergence of the projected power method for solving Problem \eqref{MLE} is to identify a proper initial point $\bH^0$, which constitutes another ingredient in our approach. Specifically, the initial point $\bH^0$ is required to satisfy the following condition: % , which denotes a partition of $n$ nodes into $K$ communities,
\begin{align}\label{init-partial}
\bH^0 \in \M_{n,K}\quad \st\ \min_{\bQ \in \Pi_K}\|\bH^0 - \bHs\bQ \|_F \le \theta\sqrt{n}, 
\end{align}
where $\theta$ is a constant that will be specified later. We remark that the condition \eqref{init-partial} is equivalent to that $\bH^0$ satisfies a partial recovery condition; see, e.g., \citet[Definition 4]{abbe2017community}. 

We now summarize the proposed method for solving Problem \eqref{MLE} in Algorithm \ref{alg:PGD}. It starts with an initial point $\bH^0$ satisfying \eqref{init-partial} and projects $\bH^0$ onto $\mH$ to make the partition balanced. Then, it refines the iterates via projected power iterations $N$ times, where $N$ is an input parameter of the algorithm, % successively until $\bH^k=\bH^{k-1}$ for some $k$, at which point it terminates and outputs $\bH^k$. 
and outputs $\bH^{N+1}$.

\begin{algorithm}[!htbp]
	\caption{Projected Power Method for Solving Problem \eqref{MLE}}  
	\begin{algorithmic}[1]  
		\STATE \textbf{Input:} adjacency matrix $\bA$, positive integer $N$
		\STATE \textbf{Initialize} an $\bH^0$ satisfying \eqref{init-partial} % set $\bH^0 \leftarrow \mT(\hat{\bU})$, where $\hat{\bU} \in \R^{n\times K}$ consists of approximate $K$ leading
		 % eigenvectors of $\bA$ satisfying \eqref{init-condition}
		\STATE set $\bH^1 \leftarrow \mT(\bH^0)$ 		
		\FOR{$k=1,2,\dots,N$}
		\STATE set $\bH^{k+1}\leftarrow \mT(\bA\bH^k)$
%		\IF{$\bH^{k+1} = \bH^k$}
%		\STATE terminate and return $\bH^{k+1}$
%		\ENDIF
		\ENDFOR
		\STATE \textbf{Output} $\bH^{N+1}$
	\end{algorithmic}
	\label{alg:PGD}
\end{algorithm}

We next present the main theorem of this paper, which shows that Algorithm \ref{alg:PGD} achieves exact recovery  down to the information-theoretic threshold and also provides its explicit iteration complexity bound. % More precisely, we shall show that if the initializer is properly chosen such that \eqref{init-partial} or \eqref{init-weak} holds, then the projected gradient update will terminate in a finite number of iterations and output a ground truth with high probability. 

\begin{thm}\label{thm-1}
Let $\bA$ be the adjacency matrix of a realization of the random graph generated according to the symmetric SBM with parameters $(n,K,p,q)$ and a planted partition $\bHs \in \H_{n,K}$. Suppose that $p,q$ satisfy \eqref{p-q} with $\sqrt{\alpha}-\sqrt{\beta} > \sqrt{K}$ and $n$ is sufficiently large. Then, there exists a constant $\gamma>0$, whose value depends only on $\alpha$, $\beta$, and $K$, such that the following statement holds with probability at least $1-n^{-\Omega(1)}$: If the initial point satisfies the partial recovery condition in \eqref{init-partial} such that
\begin{align}\label{theta}
\theta = \frac{1}{4}\min\left\{\frac{1}{\sqrt{K}},\frac{\gamma\sqrt{K}}{16(\alpha-\beta)}\right\},
\end{align}
Algorithm \ref{alg:PGD} outputs a true partition in $\lceil 2\log\log n \rceil+\left\lceil \frac{2\log n}{\log\log n} \right\rceil+2$  projected power iterations.
\end{thm}
Before we proceed, some remarks are in order. 
First, an $\bH^0$ satisfying \eqref{init-partial} can be found by a host of initialization procedures in existing methods. For example, \citet[Algorithm 2]{gao2017achieving} and \citet[Algorithm 2]{yun2014accurate} respectively proposed spectral clustering based initialization procedures that can obtain an $\bH^0$ satisfying
\begin{align}\label{init-weak}
\bH^0 \in \M_{n,K}\ \st \min_{\bQ \in \Pi_K}\|\bH^0 - \bHs\bQ \|_F \lesssim \sqrt{\frac{n}{\log n}}. 
\end{align}
These initializations are cheap to compute and automatically fulfill the partial recovery requirement in \eqref{init-partial} when $n$ is sufficiently large. Note that compared to our projected power method, the refinement procedures in \citet[Algorithm 1]{gao2017achieving} and \citet[Algorithm 1]{yun2014accurate} are rather complicated. Besides, we remark that \eqref{init-weak} is a condition of almost exact recovery (see, e.g., \citet[Definition 4]{abbe2017community}). It is much more stringent than \eqref{init-partial}, which is merely a condition of partial recovery. 

Second, as we show in Proposition \ref{prop:MCAP}, the projection in \eqref{project-H} is equivalent to a minimum-cost assignment problem (MCAP), which is a special linear programming (LP) problem and can be solved very efficiently; see \citet{tokuyama1995geometric}. We refer the reader to Section A.1 of the appendix for the formal definition of the MCAP. 
% Second, although it seems difficult to compute the projection in \eqref{project-H} at first sight, we treat it as a minimum-cost assignment problem (MCAP), which can be solved very efficiently; see \citet{tokuyama1995geometric}. We refer the reader to Section A.1 of the appendix for the formal definition of the MCAP. Remark that this problem is a special instance of the Hitchcock–Koopmans transportation problem. 
% Remark that this problem is essentially a special case of the Hitchcock transportation problem; see, e.g., \citet{hitchcock1941distribution,rao2019engineering,tokuyama1995efficient}. 

% Indeed, since each row of $\bH \in \mH$ has exactly one $1$ and $(K-1)$ $0$'s, we can show that Problem \eqref{project-H} essentially coincides with a MCAP, which can be solved very efficiently. 
\begin{prop}\label{prop:MCAP}
% For any $\bG\in \R^{n\times K}$, Problem \eqref{project-H} is equivalent to a minimum-cost $\bpi$-assignment problem with the cost matrix being $-\bG$ and $\bpi=m\bo_K$. Then, it can be solved in $\mO(K^2n\log n)$ time. 
% For any $\bG\in \R^{n\times K}$, 
Problem \eqref{project-H} is equivalent to a minimum-cost assignment problem, which can be solved in $\mO(K^2n\log n)$ time. 
\end{prop}
This, together with the time complexity of computing the matrix product $\bA\bH$ for some $\bH \in \R^{n\times K}$ and Theorem \ref{thm-1}, immediately implies the time complexity of Algorithm \ref{alg:PGD} with a qualified initialization. 
\begin{coro}\label{coro:time}
Consider the setting of Theorem \ref{thm-1}. If Algorithm \ref{alg:PGD} uses an initial point that satisfies the partial recovery condition in \eqref{init-partial} with $\theta$ in \eqref{theta}, then it outputs a true partition in 
$$\mO\left(\left(K^2+3\alpha+3(K-1)\beta\right)\frac{n\log^2n}{\log\log n}\right)$$ 
time with probability at least $1-n^{-\Omega(1)}$. 
\end{coro} 
% Finally, let us clarify the connections and differences between this work and \citet{wang2020nearly}. In terms of the connections, the proposed method in Algorithm \ref{alg:PGD} can be viewed as an extension of that in \citet{wang2020nearly}, both of which are essentially the projected gradient method applied to the corresponding ML formulation. In particular, when $K=2$, the projection operators in these two works both admit a closed-form solution, which can be done via partial sorting. In terms of the differences, our method can be applied to do community detection in the setting of multiple communities, i.e., $K \ge 2$, while that in \citet{wang2020nearly} only works when $K=2$. Besides, the method in \citet{wang2020nearly} requires a spectral initialization to satisfy a condition of almost exact recovery. By contrast, any point satisfying the partial recovery condition in \eqref{init-partial},  including some spectral initializations, is a qualified initialization for Algorithm \ref{alg:PGD}.
Finally, it is worth noting that the proposed method in Algorithm \ref{alg:PGD} can be viewed as an extension of that in \citet{wang2020nearly}, both of which are essentially the projected gradient method applied to the corresponding ML formulation. In particular, when $K=2$, the projection operators in these two works both admit a closed-form solution, which can be done via partial sorting. Moreover, our method can be applied to do community detection in the setting of multiple communities, i.e., $K \ge 2$, while that in \citet{wang2020nearly} only works when $K=2$. Besides, the method in \citet{wang2020nearly} requires a spectral initialization to satisfy a condition of almost exact recovery. By contrast, any point satisfying the partial recovery condition in \eqref{init-partial},  including some spectral initializations, is a qualified initialization for Algorithm \ref{alg:PGD}.
% Finally, compared to the information-theoretic limit for exact recovery in the symmetric SBM, i.e., $\sqrt{\alpha}-\sqrt{\beta} > \sqrt{K}$, there is a $(\sqrt{2}-1)$-factor gap between this limit and that in Theorem \ref{thm-1}. It is worth noting that for the special case of $K=2$, the proposed method achieves exact recovery at the information-theoretical limit; see the remark of Lemma \ref{lem:block-gap}.
% Finally, compared to the information-theoretic limit for exact recovery in the symmetric SBM, i.e., $\sqrt{\alpha}-\sqrt{\beta} > \sqrt{K}$, there is a $(\sqrt{2}-1)$-factor gap between this limit and that in Theorem \ref{thm-1}. However, according to our proofs, we found that there is no gap when $K=2$. Then, our result in Theorem \ref{thm-1} can recover that in \citet[Theorem 1]{wang2020nearly} if Algorithm \ref{alg:PGD} is initialized with their proposed power method. 

\section{Proofs of Main Results}\label{sec:pf-main}

In this section, we provide the proofs of our main results in Section \ref{sec:preli}. The complete proofs of the theorem, propositions, and lemmas can be found in Sections B, C of the appendix. 

\subsection{Analysis of the Projected Power Iteration}\label{sub-sec:pgd}
In this subsection, we study the convergence behavior of the projected power iterations in Algorithm \ref{alg:PGD}. Our main idea is to show the contraction property of the projection operator $\mT$ in the symmetric SBM. Let 
$$\mP = \{\bH\in\R^{n\times K}: \bH\bo_K = \bo_n, \bH^T\bo_n=m\bo_K, \bH \ge 0  \}.$$ 
To begin, we present a lemma that establishes an equivalence among the set of extreme points of this polytope, the discrete set $\mH$, and the collection of all balanced clustering matrices $\H_{n,K}$.   %relationship between this polytope the discrete set $\mH$.

\begin{lemma}\label{lem:form-H}
The following statements are equivalent: \\
(i) $\bH \in \mH$.\qquad\qquad (ii) $\bH$ is an extreme point of $\mP$. \\
(iii) $\bH \in \H_{n,K}$. 
\end{lemma}
It is worth noting that the proof this lemma builds on the total unimodularity (see, e.g., \citet{heller1956extension,hoffman2010integral}) of the equality constraint matrix of the polytope $\mP$. Equipped with this lemma, we can show that Problem \eqref{project-H} is equivalent to an LP. 
\begin{prop}\label{prop:LP}
For any $\bC \in \R^{n\times K}$, Problem \eqref{project-H} is equivalent to the following LP:
\begin{align}\label{LP-H}
\mT(\bC) = \argmax\left\{ \langle \bC,\bH \rangle:\ \bH \in \mP \right\}.
\end{align}
\end{prop}

Next, we characterize the optimal solutions of the LP in \eqref{LP-H} explicitly by exploiting the structure of the polytope $\mP$.
\begin{lemma}\label{lem:closed-form-LP}
For a matrix $\bC \in \R^{n\times K}$, it holds that $\bH \in \mT(\bC)$ if and only if 
\begin{align*}
h_{ik} = 
\begin{cases}
1,\ \text{if}\ i \in \mI_k, \\
0,\ \text{otherwise},
\end{cases}
\end{align*}
where $\mI_1,\dots,\mI_K$ satisfies (i)  ${\cup}_{k=1}^K\mI_k = [n]$, $\mI_k \cap \mI_\ell = \emptyset$, and $|\mI_k| = m$ for all $1\le k \neq \ell \le K$, and (ii) there exists $\bw \in \R^K$ such that
\begin{align}\label{form-C}
c_{ik} - c_{i\ell}  \ge w_k - w_\ell \ge  c_{jk} - c_{j\ell}
\end{align}
for all $i \in \mI_k$, $j\in \mI_\ell$, and $1\le k \neq \ell \le K$.
\end{lemma}
When $K=2$, let $\bc_1$ and $\bc_2$ denote the first and second columns of $\bC \in \R^{n\times 2}$, respectively. In this scenario, Lemma \ref{lem:closed-form-LP} implies that solving the LP in \eqref{LP-H} boils down to finding the indices that correspond to the $n/2$ largest entries of the vector $\bc_1-\bc_2$, which can be done via median finding efficiently.  

Based on the above lemma, we can show that the projection operator $\mT$ in \eqref{project-H} possesses a Lipschitz-like property in spite of the fact that $\mH$ is a discrete set. 
\begin{lemma}\label{lem:LP-cont}
Let $\delta  > 0$, $\bC \in \R^{n\times K}$  be arbitrary and $m=n/K$. Suppose that there exists a family of index sets $\mI_1,\dots,\mI_K$ satisfying ${\cup}_{k=1}^K\mI_k = [n]$, $\mI_k \cap \mI_\ell = \emptyset$, and $|\mI_k| = m$ such that $\bC$ satisfies
\begin{align}\label{eq:LP-cont-cond}
c_{ik} - c_{i\ell} \ge \delta
\end{align}
for all  $i \in \mI_k$ and $1 \le k \neq \ell \le K$.
%and 
%\begin{align}\label{eq:LP-cont-cond-1}
%c_{ik} - c_{i\ell} + c_{j\ell} - c_{jk} \ge \delta^\prime
%\end{align}
%for all $i \in \mI_k$, $j \in \mI_{\ell^\prime}$, $1 \le k \neq \ell \le K$, and $1 \le k \neq \ell^\prime  \le K$. 
Then, for any $\bV \in \mT(\bC)$, $\bC^\prime \in \R^{n\times K}$, and $\bVp \in \mT(\bCp)$, it holds that 
\begin{align}\label{rst:LP-cont}
\|\bV - \bVp\|_F \le \frac{2\|\bC-\bCp\|_F}{\delta}. 
\end{align}
\end{lemma}

Next, we show an inequality that is useful in establishing the contraction property of the projected power iterations. % for $\bA$ provided that one point is exactly a ground truth and another point is around it. 
\begin{lemma}\label{lem:contra}
Let $\Delta =\bA - \E[\bA]$. Suppose that $\varepsilon \in (0,1/\sqrt{K})$ and $\bH \in \mH$ such that $\|\bH-\bHs\bQ\|_F \le \varepsilon\sqrt{n}$ for some $\bQ \in \Pi_K$. Then, it holds that
\begin{align*}
\|\bA(\bH-\bHs\bQ)\|_F \le \left(\frac{4\varepsilon n}{\sqrt{K}}(p-q) + \|\Delta\| \right) \|\bH-\bHs\bQ\|_F.
\end{align*}  
\end{lemma}

Then, we present some probabilistic results that will be used for establishing the contraction property of the projected power iterations. 

\begin{lemma}\label{lem:spectral-norm-Delta}
Let $\Delta =\bA - \E[\bA]$.  There exists a constant $c_1>0$, whose value only depends on $\alpha$ and $\beta$, such that 
\begin{align}\label{eq:spectral-norm-Delta}
\|\Delta\| \le c_1\sqrt{\log n}
\end{align}
holds with probability at least $1-n^{-3}$. 
\end{lemma}
This lemma provides a spectral bound on the deviation of $\bA$ from its mean. It is a direct consequence of \citet[Theorem 5.2]{lei2015consistency} and thus we omit its proof. 

\begin{lemma}\label{lem:tail-Bino}% , $\{Z_i^\prime\}_{i=1}^m$, and $\{Z_i^{\prime\prime}\}_{i=1}^m$ $\{Z_i^\prime\}_{i=1}^m$, and $\{Z_i^{\prime\prime}\}_{i=1}^m$
Let $m=n/K$ and $\alpha > \beta >0$ be constants. Suppose that $\{W_i\}_{i=1}^m$ are i.i.d.~$\mathbf{Bern}(\alpha\log n/n)$ and $\{Z_i\}_{i=1}^m$ are i.i.d.~$\mathbf{Bern}(\beta\log n/n)$ that is independent of $\{W_i\}_{i=1}^m$. Then, for any $\gamma \in \R$, it holds that 
\begin{align*}% \label{tail-Bino-1}
 \P\left( \sum_{i=1}^m W_i - \sum_{i=1}^m Z_i \le \gamma \log n \right) \le n^{-\frac{(\sqrt{\alpha}-\sqrt{\beta})^2}{K}+\frac{\gamma \log(\alpha/\beta)}{2}}.
\end{align*}
%and
%\begin{align}\label{tail-Bino-2}
%& \P\left( \sum_{i=1}^mW_i - \sum_{i=1}^m Z_i + \sum_{i=1}^mZ_i^{\prime} - \sum_{i=1}^m Z_i^{\prime \prime} \le \eta \log n \right) \notag \\
%&\qquad \qquad \qquad \le n^{-\frac{(\sqrt{\alpha+\beta}-\sqrt{2\beta})^2}{K}+\frac{\eta\log\left((\alpha+\beta)/(2\beta)\right)}{2}}. 
%\end{align}
\end{lemma}
%In this lemma, the above two probabilistic results respectively present upper bounds for the tails of the difference of Binomial variables and the sum of differences of Binomial variables. 
This lemma is proved in \citet[Lemma 8]{abbe2020entrywise}. Based on the this lemma, we can show that the entries of $\bA\bHs$ satisfy the requirement of \eqref{eq:LP-cont-cond} in Lemma \ref{lem:LP-cont} with high probability.  % and \eqref{eq:LP-cont-cond-1} in Lemma \ref{lem:LP-cont} 
\begin{lemma}\label{lem:block-gap}
Suppose that $\alpha > \beta > 0$ and $\bC=\bA\bHs$. Let $\mI_k=\{i \in [n]: h_{ik}^* =1\}$ for all $k\in [K]$. If $\sqrt{\alpha}-\sqrt{\beta} > \sqrt{K}$, there exists a constant $\gamma> 0$, whose value depends only on $\alpha$, $\beta$, and $K$, such that for all $i \in \mI_k$ and $1\le k \neq \ell \le K$,
\begin{align}\label{rst-1:lem-block-gap}
c_{ik} - c_{i\ell} \ge \gamma\log n
\end{align}
holds with probability at least $1-n^{-\Omega(1)}$. 
%Moreover, if $\sqrt{\alpha}-\sqrt{\beta} > \sqrt{2K}$, there exists a constant $\eta > 0$, whose value depends only on $\alpha$, $\beta$, and $K$, such that for all $i \in \mI_k$, $j \in \mI_{\ell^\prime}$, $1 \le k \neq \ell \le K$, and $1 \le k \neq \ell^\prime  \le K$,
%\begin{align}\label{rst-2:lem-block-gap}
%c_{ik} - c_{i\ell} + c_{j\ell} - c_{jk} \ge \eta\log n
%\end{align}
%holds with probability at least $1-n^{-\Omega(1)}$.
\end{lemma}
% Remark that when $K=2$, we have $\ell=\ell^\prime$, and thus \eqref{rst-1:lem-block-gap} implies \eqref{rst-2:lem-block-gap} with $\eta=2\gamma$ directly. In this setting, we found that the proposed method achieves exact recovery all the way down to the information-theoretic limit according to the subsequent proofs. 

Armed with the above results, we are now ready to show that the projected power iteration possesses a contraction property in a certain neighborhood of $\bHs\bQ$ for some $\bQ \in \Pi_K$. 

\begin{prop}\label{prop:contra-PGM}
Suppose that the constants $\alpha, \beta > 0$ satisfy $\sqrt{\alpha}-\sqrt{\beta} > \sqrt{K}$ and $n > \exp(16c_1^2/\gamma^2)$. Then, the following event happens with probability at least $1-n^{-\Omega(1)}$: For all $\bH \in \mH$ and  $\varepsilon \in \left(0,\min\left\{\frac{1}{\sqrt{K}},\frac{\gamma\sqrt{K}}{16(\alpha-\beta)}\right\}\right)$  such that $\|\bH-\bHs\bQ\|_F \le \varepsilon \sqrt{n}$ for some $\bQ \in \Pi_K$, it holds that
\begin{align}\label{rst:prop-PGM}
\|\bV - \bHs\bQ\|_F \le \kappa \|\bH-\bHs\bQ\|_F
\end{align} 
for any $\bV \in \mT(\bA\bH)$, where 
\begin{align}\label{conv-rate}
% & r = \min\left\{\frac{1}{\sqrt{K}},\frac{\eta\sqrt{K}}{32(\alpha-\beta)}\right\},\\
& \kappa = 4\max\left\{ \frac{4\varepsilon (\alpha-\beta)}{\gamma\sqrt{K}}, \frac{c_1}{\gamma\sqrt{\log n}} \right\} \in (0,1)
\end{align}
and $c_1, \gamma$ are the constants in Lemmas \ref{lem:spectral-norm-Delta} and \ref{lem:block-gap}, respectively. 
\end{prop}
Observe that the contraction rate $\kappa$ is decreasing to a quantity on the order of $1/\sqrt{\log n}$ as the iterates approach a ground truth. This implies that the better the initialization, the less iterations the proposed method requires to find a ground truth. 

The following lemma indicates that the projected power iterations exhibit one-step convergence to a ground truth. This would imply the finite termination of the proposed algorithm. % More precisely, if an iterate is sufficiently close to an $\bHs\bQ$ for some $\bQ\in \Pi_K$, then the next iterate is exactly $\bHs\bQ$ with high probability.  

\begin{lemma}\label{lem:one-step-conv}
Suppose that the constants $\alpha > \beta > 0$ satisfy $\sqrt{\alpha}-\sqrt{\beta} > \sqrt{K}$. Then, the following statement holds with probability at least $1-n^{-\Omega(1)}$: For all $\bH \in \mH$ such that $\|\bH-\bHs\bQ\|_F < \sqrt{\gamma\log n}$ for some $\bQ \in \Pi_K$, it holds that
\begin{align}\label{rst:lem-one-step-conv}
\mT(\bA\bH) = \{\bHs\bQ\},
\end{align}
where $\gamma > 0$ is the constant in Lemma \ref{lem:block-gap}. 
\end{lemma}

\begin{figure*}[!htbp]
	\begin{minipage}[b]{0.245\linewidth}
		\centering
		\centerline{\includegraphics[width=\linewidth]{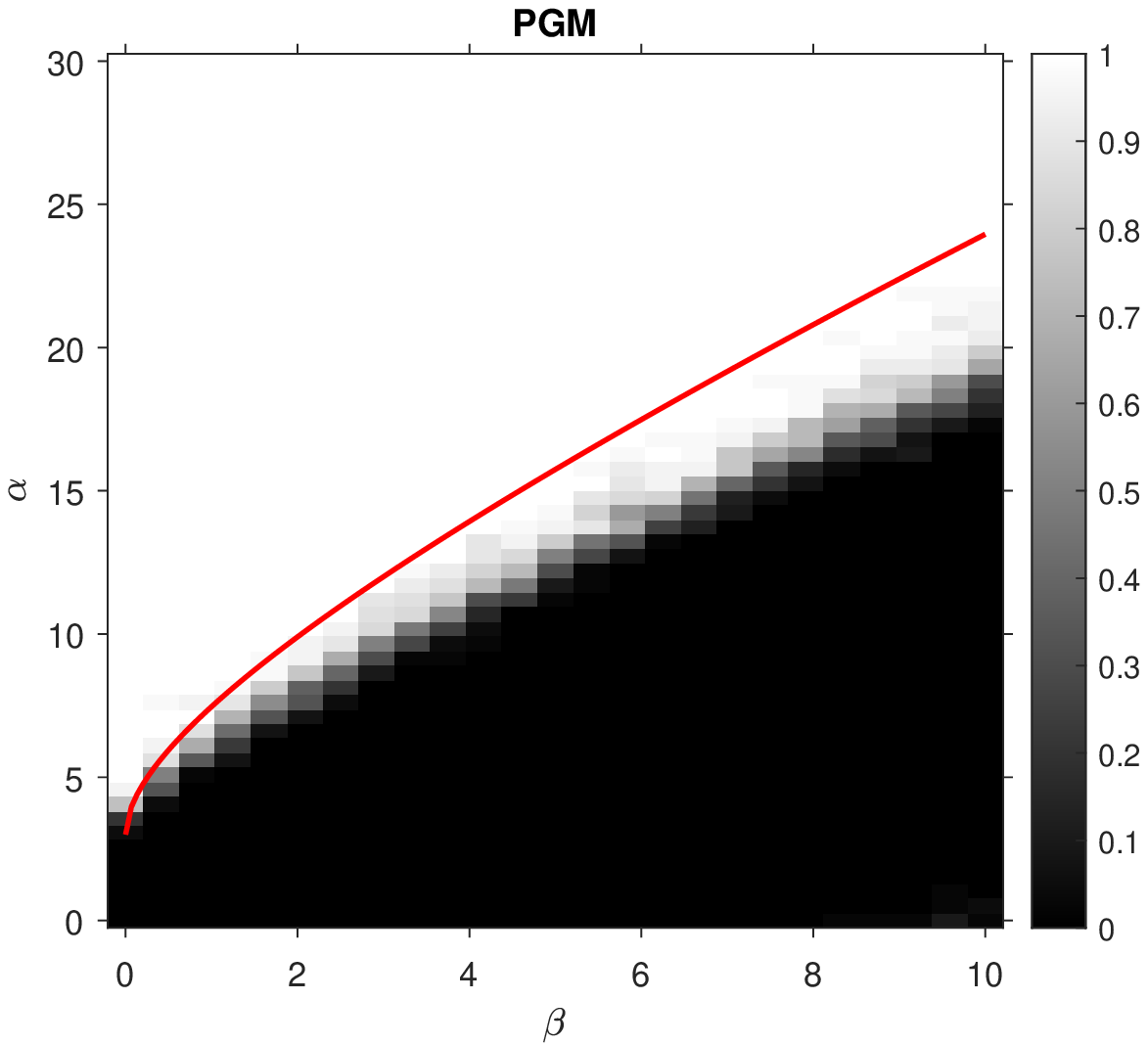}}
		%  \vspace{1.0cm}
		\centerline{(a) PPM}\medskip
	\end{minipage}
	%\hfill
	\begin{minipage}[b]{0.245\linewidth}
		\centering
		\centerline{\includegraphics[width=\linewidth]{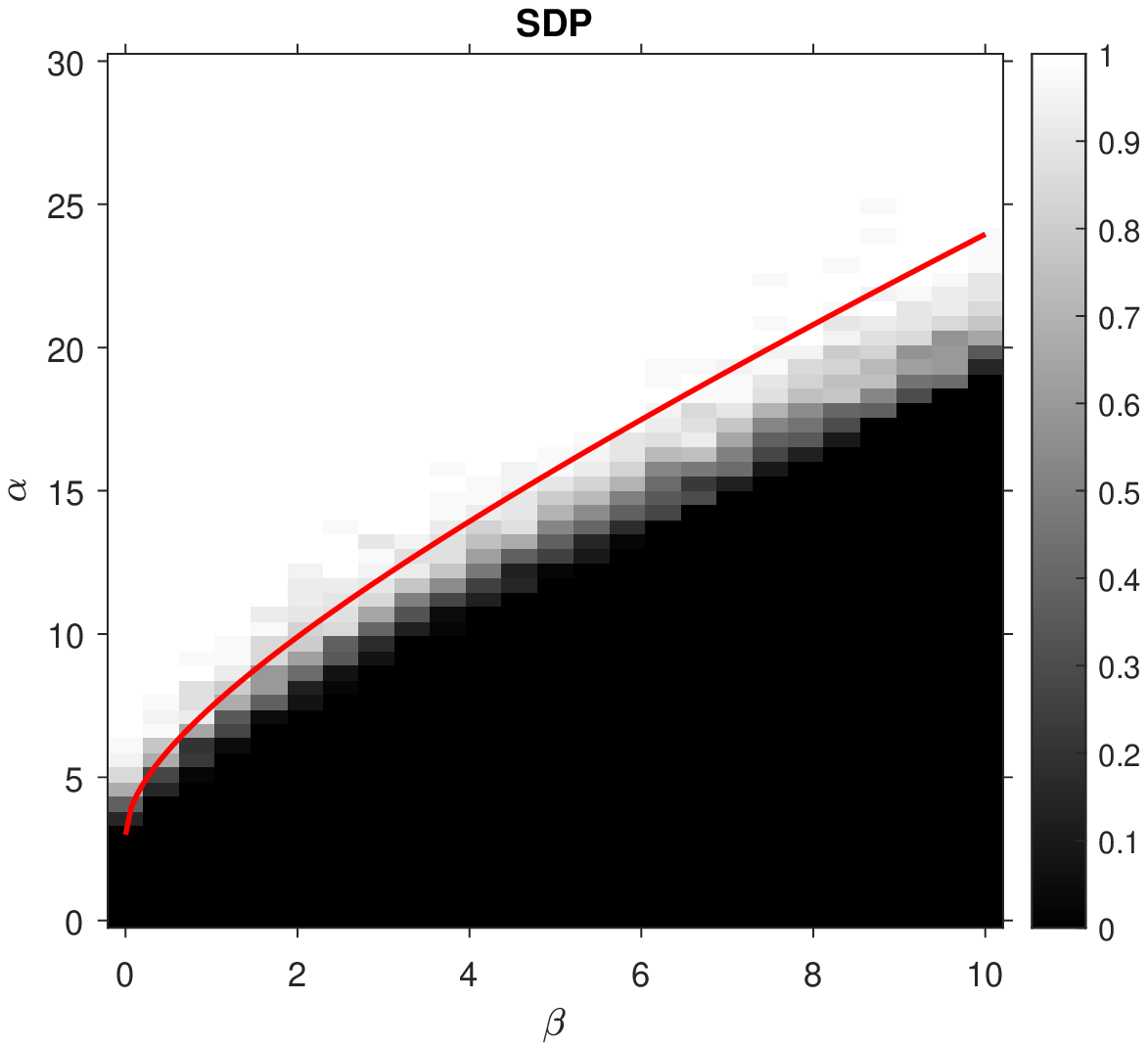}}
		%\cspace{1.5cm}
		\centerline{(b) SDP}\medskip
	\end{minipage}
	\begin{minipage}[b]{0.245\linewidth}
		\centering
		\centerline{\includegraphics[width=\linewidth]{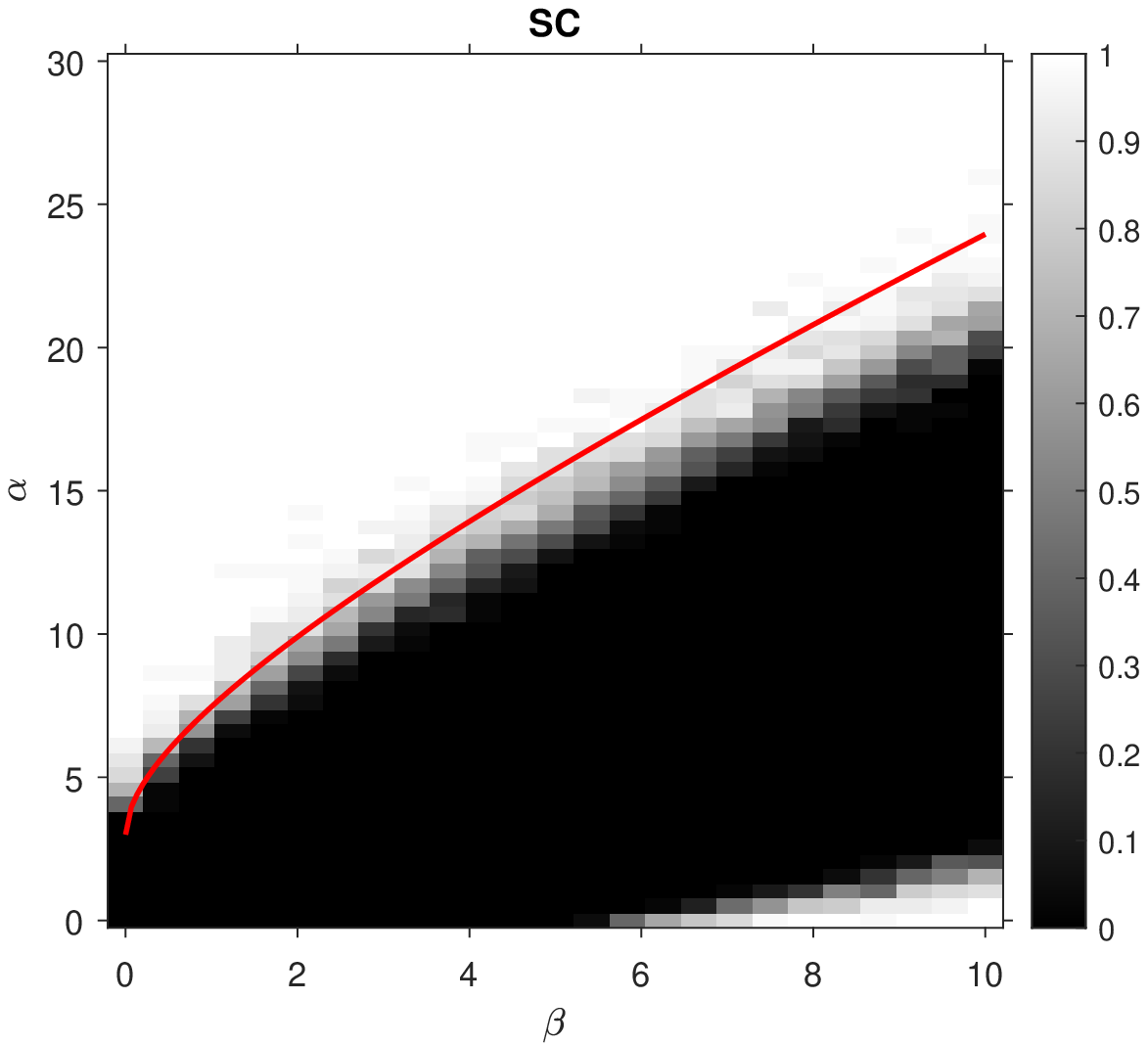}}
		%  \vspace{1.0cm}
		\centerline{(c) SC}\medskip
	\end{minipage}
	\begin{minipage}[b]{0.245\linewidth}
		\centering
		\centerline{\includegraphics[width=\linewidth]{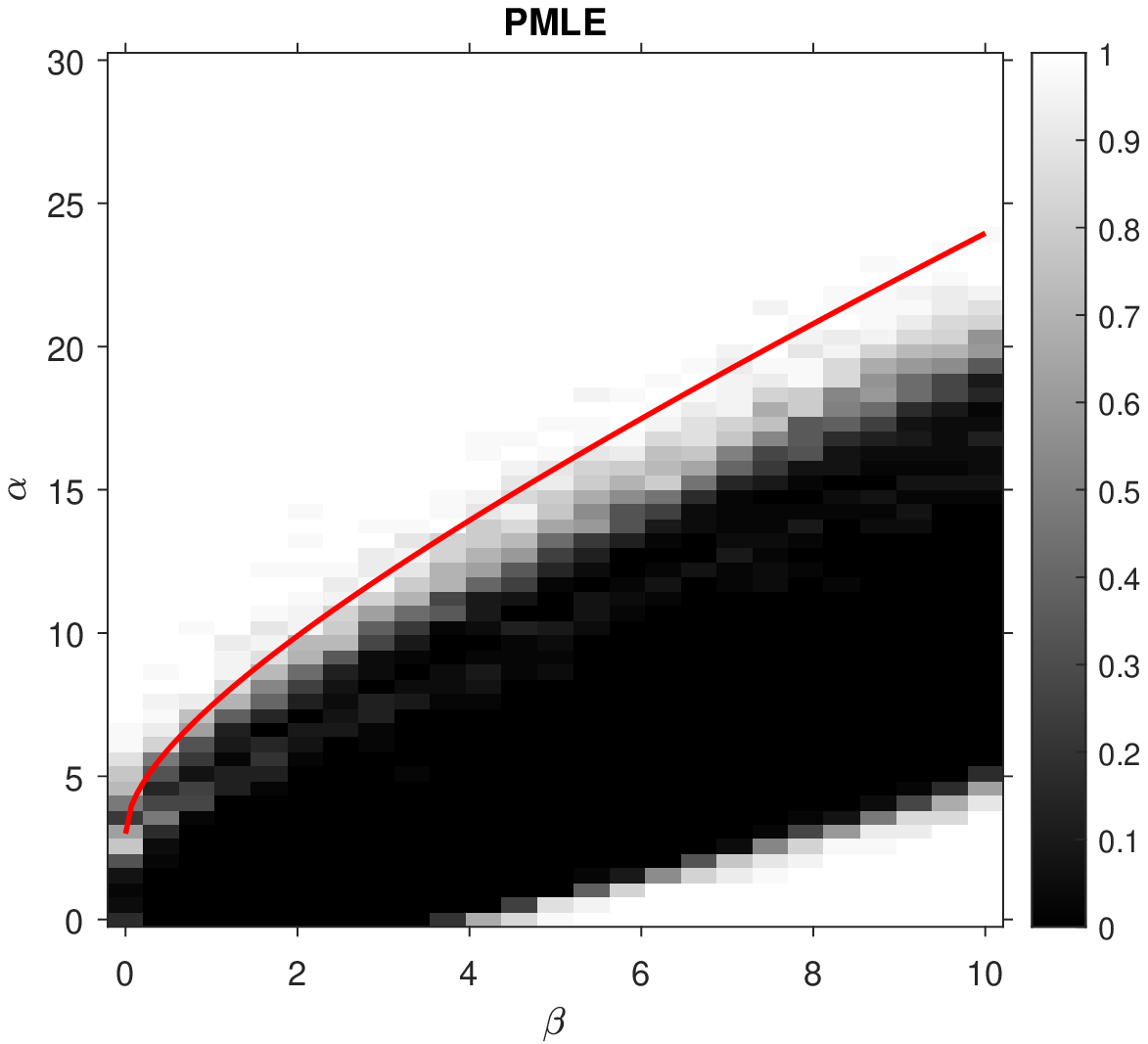}}
		%  \vspace{1.0cm}
		\centerline{(d) PMLE}\medskip
	\end{minipage}
	\vskip -0.1in
	\caption{Phase transition in the setting of $n=300,K=3$: The $x$-axis is $\beta$, the $y$-axis is $\alpha$, and darker pixels represent lower empirical probability of success. The red curve is the information-theoretic threshold $\sqrt{\alpha}-\sqrt{\beta}=\sqrt{3}$.}
	\label{fig-1}
\end{figure*}

\begin{figure*}[!htbp]
	\begin{minipage}[b]{0.245\linewidth}
		\centering
		\centerline{\includegraphics[width=\linewidth]{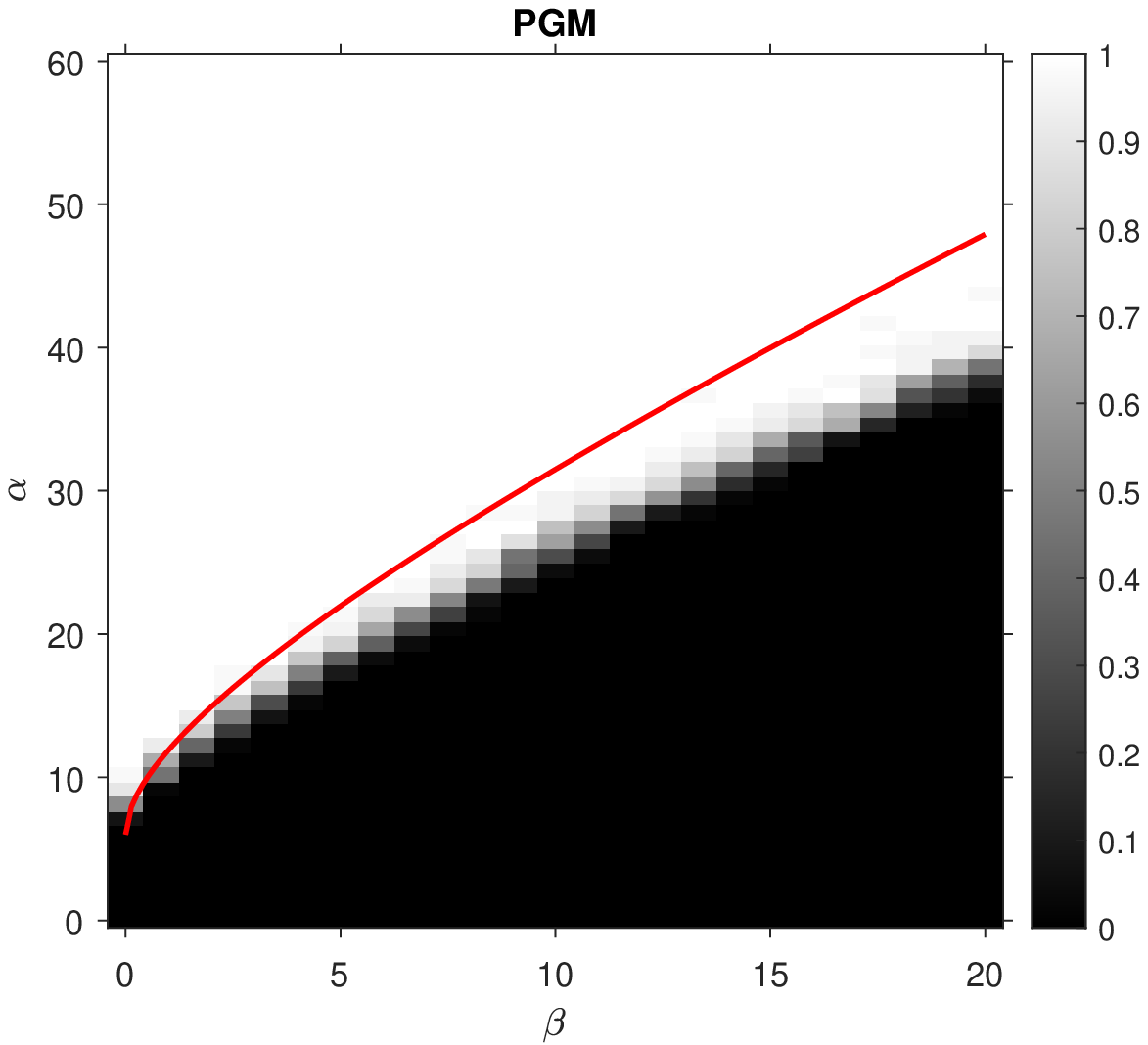}}
		%  \vspace{1.0cm}
		\centerline{(a) PPM}\medskip
	\end{minipage}
	%\hfill
	\begin{minipage}[b]{0.245\linewidth}
		\centering
		\centerline{\includegraphics[width=\linewidth]{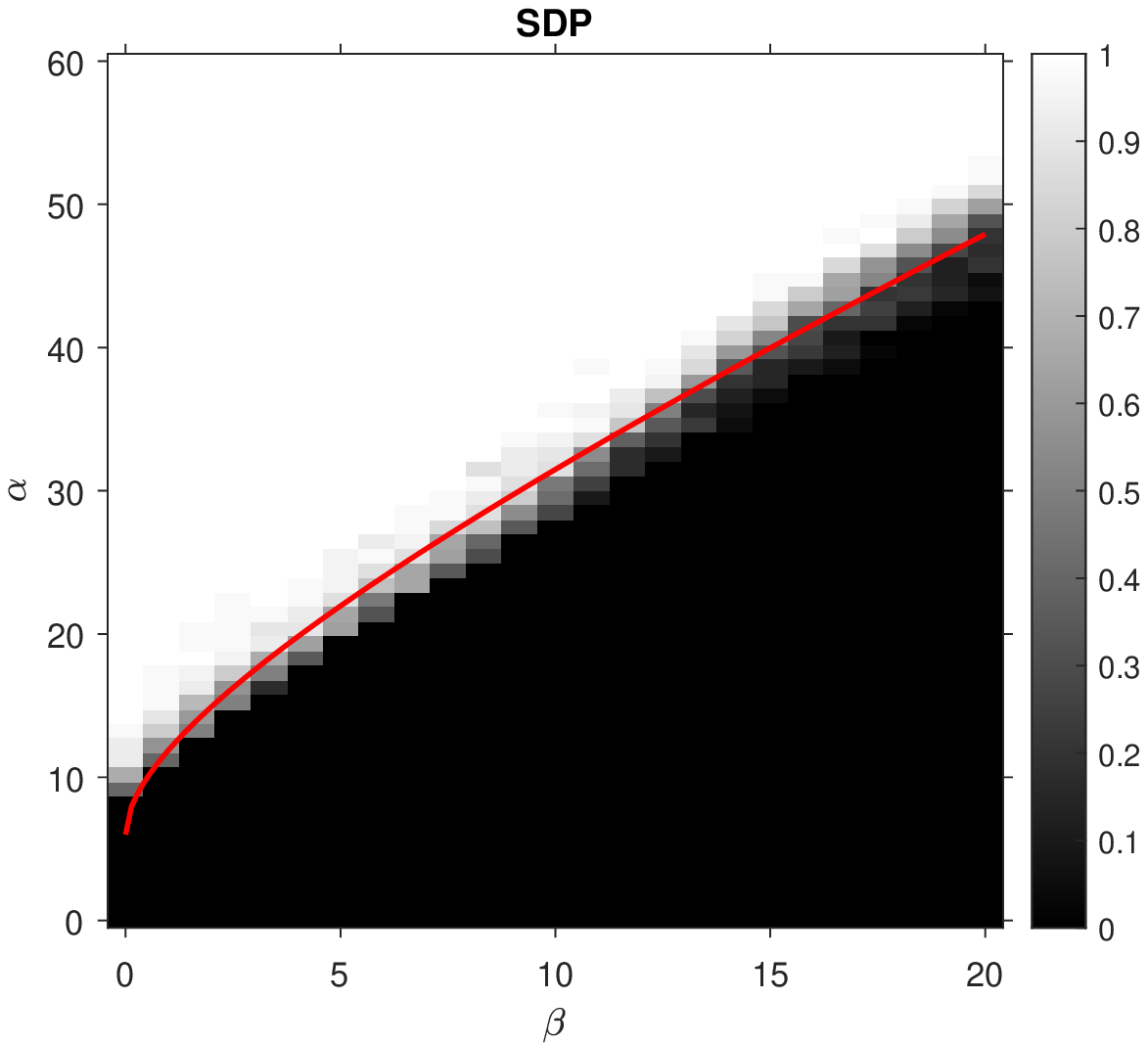}}
		%\cspace{1.5cm}
		\centerline{(b) SDP}\medskip
	\end{minipage}
	\begin{minipage}[b]{0.245\linewidth}
		\centering
		\centerline{\includegraphics[width=\linewidth]{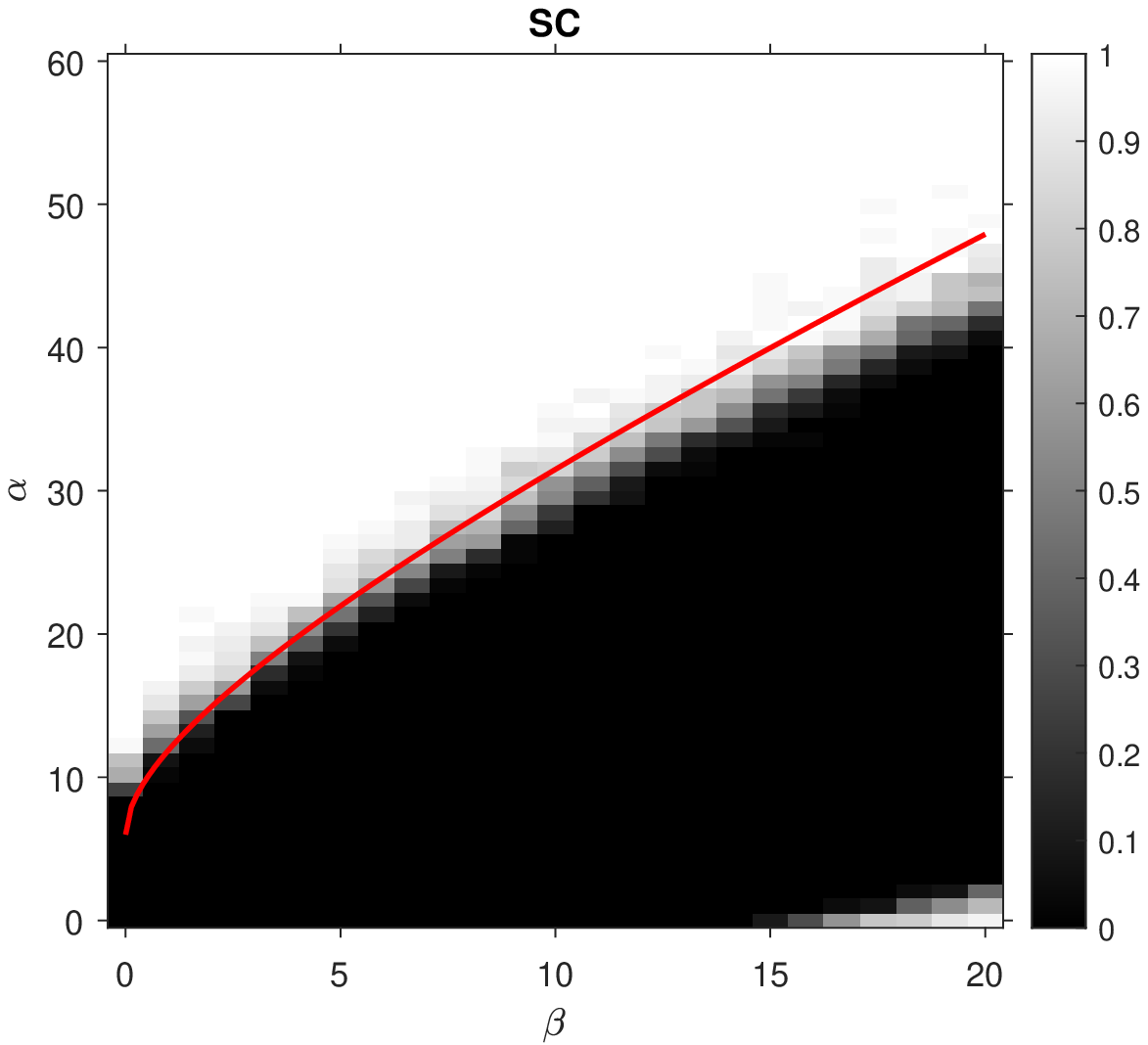}}
		%  \vspace{1.0cm}
		\centerline{(c) SC}\medskip
	\end{minipage}
	\begin{minipage}[b]{0.245\linewidth}
		\centering
		\centerline{\includegraphics[width=\linewidth]{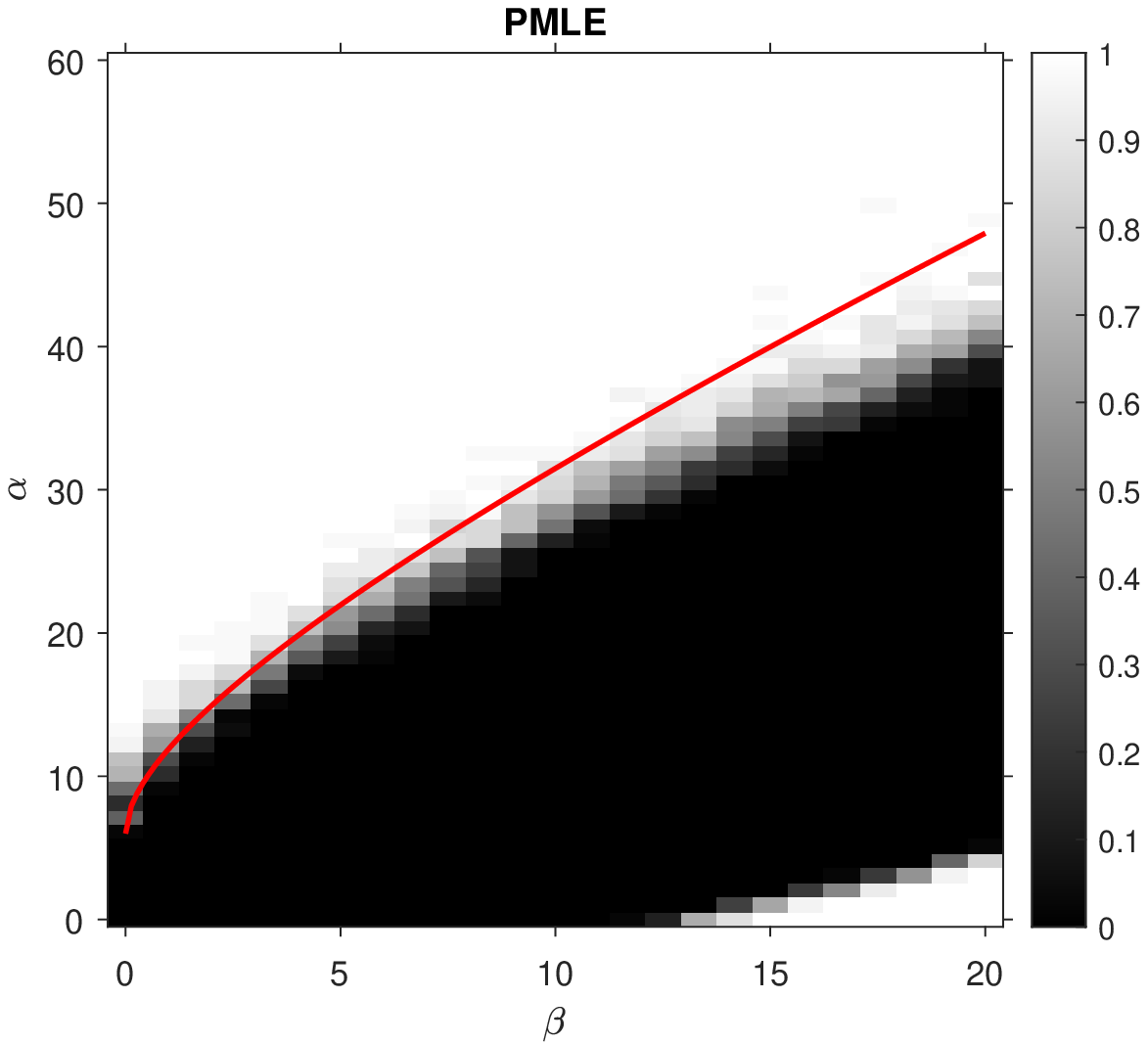}}
		%  \vspace{1.0cm}
		\centerline{(d) PMLE}\medskip
	\end{minipage}
	\vskip -0.1in
	\caption{Phase transition in the setting of $n=600,K=6$:  The $x$-axis is $\beta$, the $y$-axis is $\alpha$, and darker pixels represent lower empirical probability of success. The red curve is the information-theoretic threshold $\sqrt{\alpha}-\sqrt{\beta}=\sqrt{6}$.}
	\label{fig-2}
	% \vskip -0.15in
\end{figure*}

\subsection{Proof of Theorem \ref{thm-1}}

Now, we are ready to derive the iteration complexity bound of Algorithm \ref{alg:PGD} equipped with the results in Section \ref{sub-sec:pgd}. We first provide a formal version of Theorem \ref{thm-1} and then sketch its proof. The full proof can be found in Section C of the appendix. Recall that $\theta$, $c_1$, and $\gamma$ are the constants in Theorem \ref{thm-1}, Lemma \ref{lem:spectral-norm-Delta}, and Lemma \ref{lem:block-gap}, respectively. To simplify the notations in the sequel, let 
\begin{align}\label{phi}
\phi = \frac{c_1\sqrt{K}}{16(\alpha-\beta)}.
\end{align}

\begin{thm}\label{thm-2}
Consider the setting of Theorem \ref{thm-1}. Suppose that 
\begin{align}\label{n}
n > \exp\left(\max\left\{ \frac{64c_1^2}{\gamma^2}, \frac{\gamma^2}{c_1^2}, \frac{4\sqrt{2}\phi}{\sqrt{\gamma}}, \frac{256c_1^4}{\gamma^4} \right\}\right). 
\end{align}
Then, the following statement holds with probability at least $1-n^{-\Omega(1)}$:
If the initial point $\bH^0\in \M_{n,K}$ satisfies % the partial recovery requirement in \eqref{init-partial} such that $\theta$ 
\begin{align}\label{partial-recov}
\|\bH^0-\bHs\bQ\|_F \le  \theta\sqrt{n}
\end{align}
for some $\bQ \in \Pi_K$ and $\theta$ is defined in \eqref{theta}, Algorithm \ref{alg:PGD} outputs $\bHs\bQ$ within $\lceil 2\log\log n \rceil+\left\lceil \frac{2\log n}{\log\log n} \right\rceil+2$  projected power iterations. 
\end{thm}
\begin{proof}
Suppose that the statements in Proposition \ref{prop:contra-PGM} and Lemma \ref{lem:one-step-conv} hold, which happens with probability at least $1-n^{-\Omega(1)}$ by the union bound. % Let $\mI_k=\left\{i\in [n]:h^*_{ik}=1\right\}$ for all $k\in [K]$. %Then, one can verify that $h_{ik}^*-h_{i\ell}^* \ge 1$ for all $i\in \mI_k$ and $k\neq \ell$, and $h_{ik}^*-h_{i\ell}^*+h_{j\ell}^*-h_{jk}^* \ge 1$ for all $i\in \mI_k$, $j\in \mI_{\ell^\prime}$, $k\neq \ell$, and $k \neq \ell^\prime$. This, together with $\bH^1\in \mT(\bH^0)$ in Algorithm \ref{alg:PGD}, Lemma \ref{lem:LP-Q}, Lemma \ref{lem:closed-form-LP}, and Lemma \ref{lem:LP-cont}, yields that $\bH^1 \in \H_{n,K}$ satisfies
%\begin{align*}
%\|\bH^1 - \bHs\bQ\|_F \le 4 \|\bH^0 - \bHs\bQ\|_F. 
%\end{align*}
We first show that for all $k \ge 2$, $\bH^k \in \H_{n,K}$ satisfies  $\|\bH^k-\bHs\bQ\|_F \le 2\theta\sqrt{n}$ and
\begin{align*}
\|\bH^k - \bHs\bQ\|_F \le \frac{1}{2} \|\bH^{k-1}-\bHs\bQ\|_F,
\end{align*}
and it holds for $N_1=\lceil 2\log\log n \rceil + 1$ that
\begin{align*}
\|\bH^{N_1} - \bHs\bQ\|_F \le  2\phi\sqrt{\frac{n}{\log n}}.
\end{align*}
Next, we show that for all $k \ge 1$, $\bH^{N_1+k} \in \H_{n,K}$ satisfies $\|\bH^{N_1+k}-\bHs\bQ\|_F \le 2\phi\sqrt{n/\log n}$ and
\begin{align*}
\|\bH^{N_1+k} - \bHs\bQ\|_F \le \frac{4c_1}{\gamma\sqrt{\log n}} \|\bH^{N_1+k-1}-\bHs\bQ\|_F,
\end{align*}
and it holds for $N_2=\left\lceil \frac{2\log n}{\log\log n} \right\rceil$ that 
\begin{align*}
\|\bH^{N_2+N_1} - \bHs\|_F < \sqrt{\gamma\log n}. 
\end{align*}
%Since $n\ge \exp\left(\phi^2/\theta^2\right)$ and $r=4\theta$, it holds that $4\phi/\sqrt{\log n} \le r$. This, together with $\bH^{N_1}\in\H_{n,K}$, $\bH^{N_1+1}\in\mT(\bA\bH^{N_1})$, Proposition \ref{prop:contra-PGM}, and \eqref{phi}, yields
%\begin{align*}
%&\qquad \|\bH^{N_1+1}-\bHs\|_F \\
%& \le 8\max\left\{ \frac{16\phi(\alpha-\beta)}{\eta\sqrt{K\log n}},\frac{c_1}{\sqrt{\log n}} \right\}\|\bH^1-\bHs\|_F \\
%& \le \frac{8c_1}{\sqrt{\log n}}\|\bH^1-\bHs\|_F.
%\end{align*}
%Then, \eqref{step2:pf-thm-1} holds for $k=1$. We can show that \eqref{step2:pf-thm-1} holds for $k\ge 2$ by an inductive argument. Thus, \eqref{step2:pf-thm-1} can be established by an induction method. Then, let $N_2=\left\lceil \frac{3\log n}{\log\log n} \right\rceil$. According to $n\ge \exp\left(4\sqrt{2}\phi/\sqrt{\gamma}\right)$ and $n \ge \exp\left( 512c_1^3\right)$, we have $4\phi/\sqrt{\log n} \le \sqrt{\gamma\log n/2}$ and $\log\log n\ge 3\log(8c_1)$. This, together with \eqref{step2:pf-thm-1}, yields
%\begin{align*}
%& \|\bH^{N_1+N_2} - \bHs\|_F  \le   \|\bH^{N_1}-\bHs\|_F\left(\frac{8c_1}{\sqrt{\log n}} \right)^{\left\lceil \frac{3\log n}{\log\log n} \right\rceil} \\
%& \le 4\phi\sqrt{\frac{n}{\log n}} \left( \frac{8c_1}{\sqrt{\log n}} \right)^{\frac{3\log n}{\log\log n}} \\
%& \le 4\phi\sqrt{\frac{n}{\log n}} \left( \frac{8c_1}{\sqrt{\log n}} \right)^{\frac{\log n}{\log\log n-2\log(8c_1)}}\\
%&  = \frac{4\phi}{\sqrt{\log n}} \le  \sqrt{\frac{\gamma\log n}{2}}.
%\end{align*}
%Thus, \eqref{step2:pf-thm-1} holds for $N_2=\left\lceil \frac{3\log n}{\log\log n} \right\rceil$.
Once this holds, we have $\bH^{N_1+N_2+1}=\bHs$  by Lemma \ref{lem:one-step-conv}. Then, the desired result is established. % This, together with Lemma \ref{lem:one-step-conv}, yields $\bH^{N_1+N_2+2}=\bHs$. According to the stopping criterion in Algorithm \ref{alg:PGD}, the desired result is established. 
\end{proof}

\section{Experimental Results}\label{sec:num}

\begin{figure*}[t]
\begin{center}
	\begin{minipage}[b]{0.33\linewidth}
		\centering
		\centerline{\includegraphics[width=\linewidth]{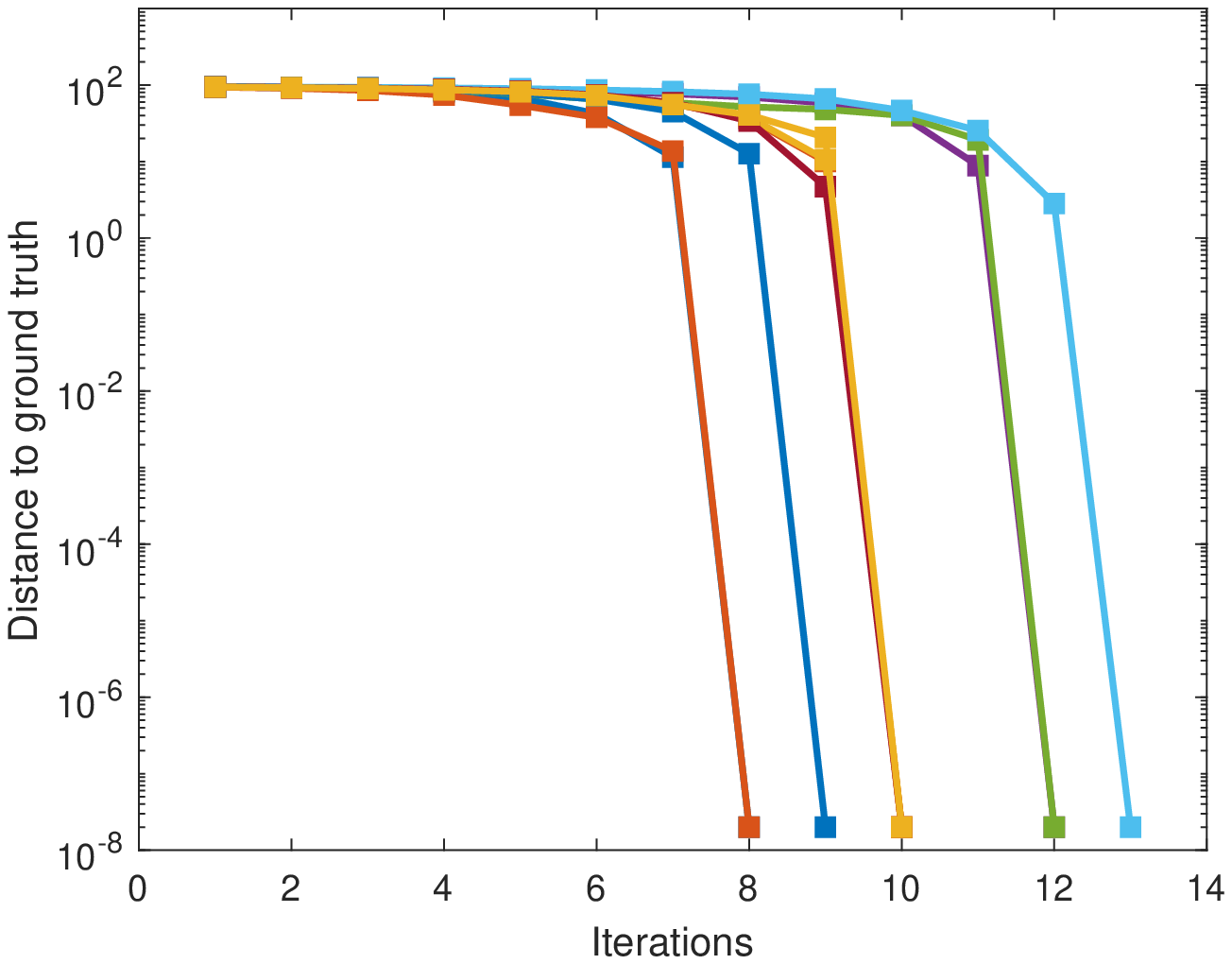}}
		%  \vspace{1.0cm}
		\centerline{(a) {\footnotesize $(\alpha,\beta,K)=(18,4,4)$}}\medskip
	\end{minipage}
	%\hfill
	\begin{minipage}[b]{0.33\linewidth}
		\centering
		\centerline{\includegraphics[width=\linewidth]{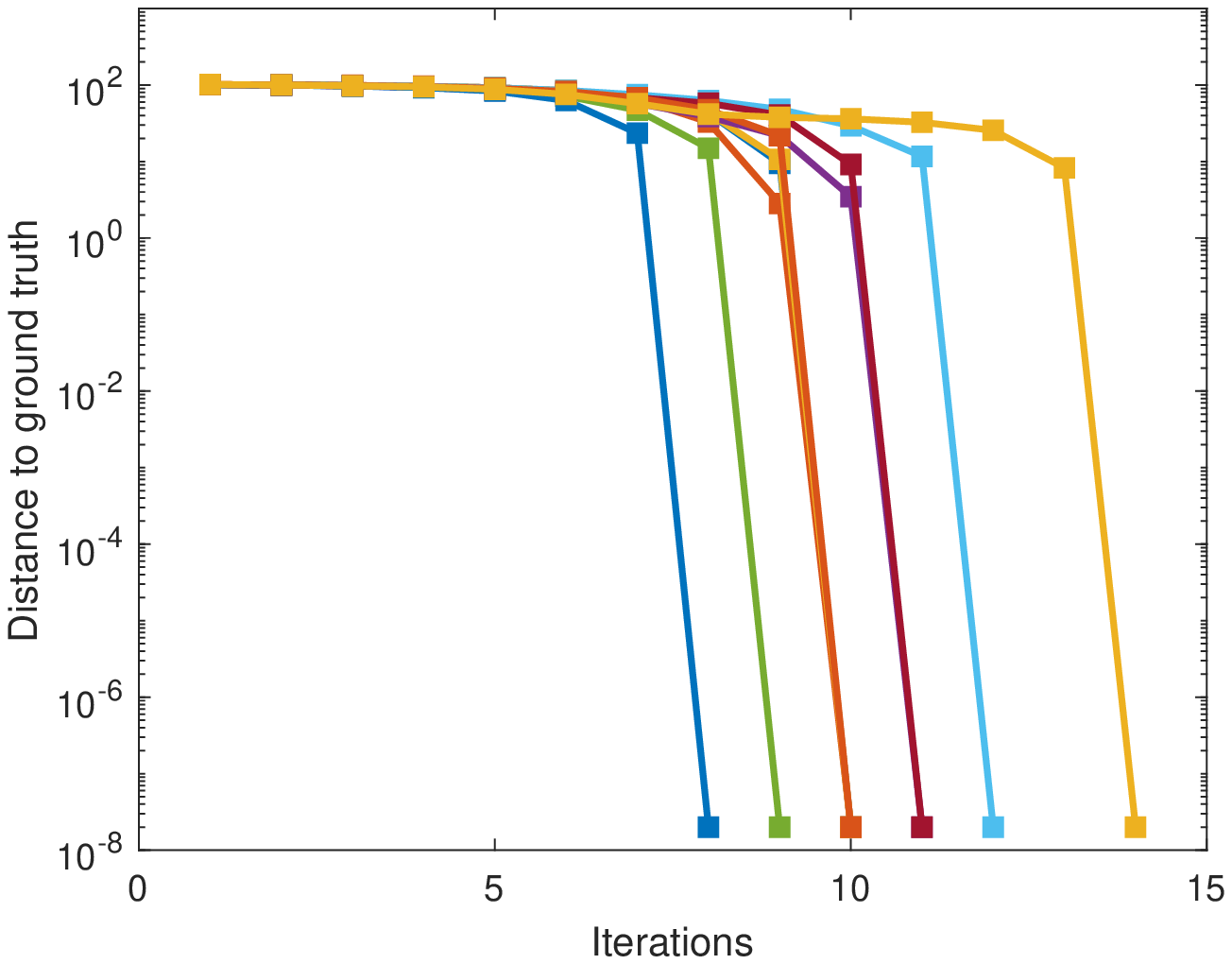}}
		%\cspace{1.5cm}
		\centerline{(b) {\footnotesize $(\alpha,\beta,K)=(36,8,8)$}} \medskip
	\end{minipage}
	\begin{minipage}[b]{0.33\linewidth}
		\centering
		\centerline{\includegraphics[width=\linewidth]{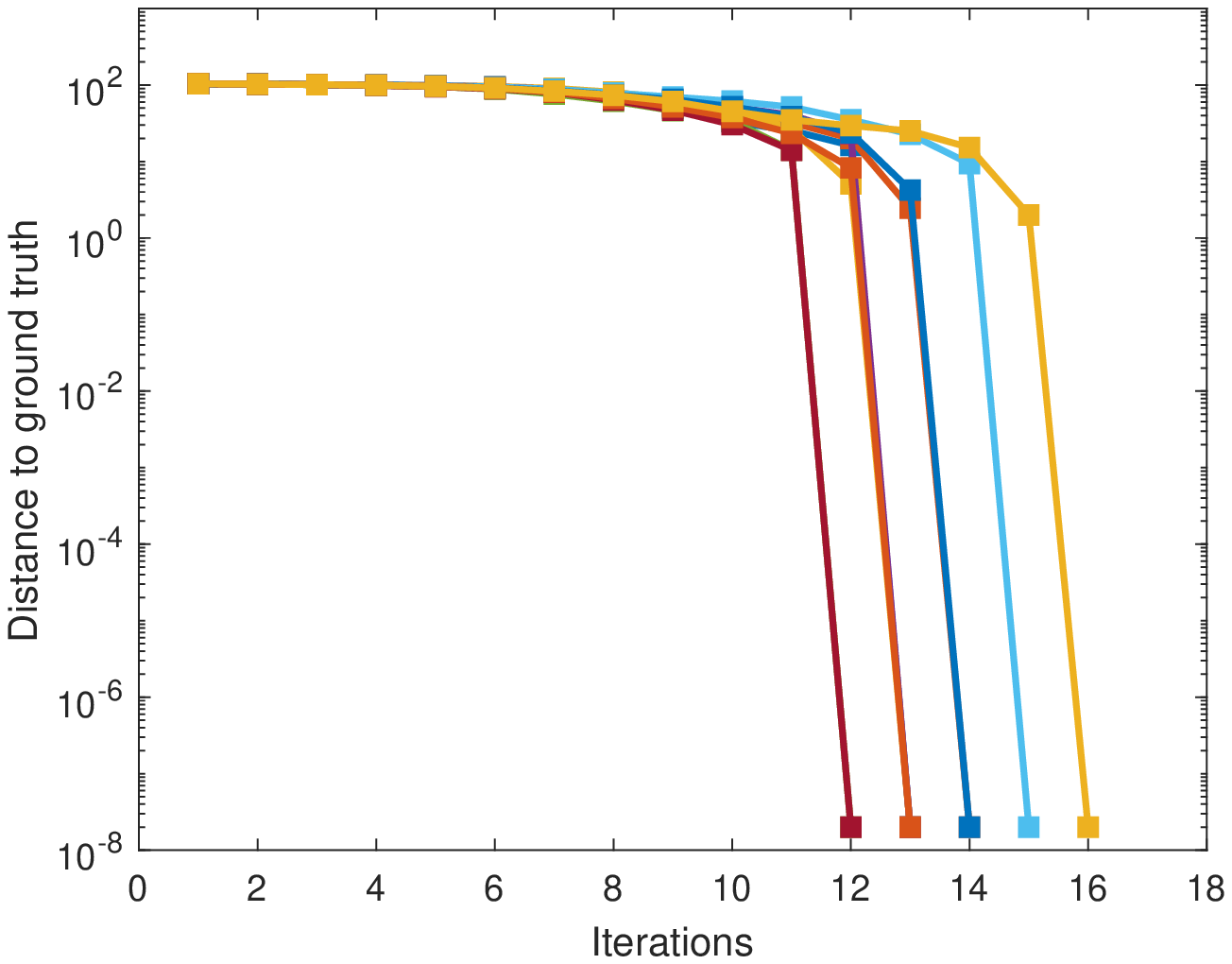}}
		%  \vspace{1.0cm}
		\centerline{(c) {\footnotesize $(\alpha,\beta,K)=(54,12,12)$}}\medskip
	\end{minipage}
\end{center}
	\vskip -0.1in
	\caption{Convergence performance of PPM: The $x$-axis is number of iterations and the $y$-axis is the distance from an iterate to a ground truth, i.e., $\min_{\bQ \in \Pi_K}\|\bH^k-\bH^*\bQ\|_F$, where $\bH^k$ is the $k$-th iterate generated by PPM.}
	\label{fig-3}
	% \vskip -0.15in
\end{figure*}

In this section, we report the recovery performance and numerical efficiency of our proposed method for recovering communities on both synthetic and real data sets. We also compare our method with three existing methods, which are the SDP-based method in \citet{amini2018semidefinite}, the spectral clustering (SC) method in \citet{su2019strong}, and the local penalized ML estimation (PMLE) method in \citet{gao2017achieving}. In the implementation, we employ \citet[Algorithm 2]{gao2017achieving} for computing the initial point $\bH^0$ in Algorithm \ref{alg:PGD} if we do not specify the initialization method. Moreover, we use alternating direction method of multipliers (ADMM) for solving the SDP as suggested in \citet{amini2018semidefinite},\footnote{The code can be downloaded at \url{https://github.com/aaamini/SBM-SDP}.} the MATLAB function \textsf{eigs} for computing the eigenvectors that are needed in the SC method and the first stage of the PMLE method, and the MATLAB function \textsf{kmeans} for computing the partition in the SC method. For ease of reference, we denote our method simply by PPM. All of our simulations are implemented in MATLAB R2020a on a PC running Windows 10 with 16GB memory and Intel(R) Core(TM) i5-8600 3.10GHz CPU. Our code is available
at \url{https://github.com/peng8wang/ICML2021-PPM-SBM}. % \footnote{The code is available at \url{https://github.com/aaamini/SBM-SDP}}

\vskip 0.25 in

\subsection{Phase Transition and Computational Time}\label{sub-sec:phase}

We first conduct the experiments to examine the phase transition property and running time of the aforementioned methods for recovering communities in graphs that are generated by the symmetric SBM in Definition \ref{SBM}. We have two sets of simulations. We choose $n=300,K=3$ (resp. $n=600,K=6$), and let the parameter $\alpha$ in \eqref{p-q} vary from $0$ to $30$ (resp. 60) with increments of $0.5$ (resp. $1$) and the parameter $\beta$ in \eqref{p-q} vary from $0$ to $10$ (resp. $20$) with increments of $0.4$ (resp. $0.8$). For every pair of $\alpha$ and $\beta$, we generate $40$ instances and calculate the ratio of exactly recovering the communities for all the tested methods. The phase transition results are reported in Figures \ref{fig-1} and \ref{fig-2}. According to these figures, we can observe that all the methods exhibit a phase transition phenomenon and the recovery performance of PPM is slightly better than the other three methods. Moreover, Figures \ref{fig-1}(a) and \ref{fig-2}(a) indicate that PPM achieves the optimal recovery threshold, which supports the result in Theorem \ref{thm-1}. Besides, we record the total CPU time consumed by each method for completing the phase transition experiments in Table \ref{table-1}. It can be observed that PPM is slightly better than PMLE and substantially faster than SC and SDP. 

\begin{table}[!htbp]
\caption{Total CPU times (in seconds) of the methods  in the phase transition experiments.}
\vskip 0.1in
\label{table-1}
\begin{center}
\begin{tabular}{ccccc}
\toprule
$\quad$ Time (s) & PPM & SDP & SC & PMLE \\
\midrule
$n=300,\ K=3$ & {\bf 401} & 25887 &  1438 & 572 \\
$n=600,\ K=6$ & {\bf 1824} & 82426 & 3669 & 2661\\
\bottomrule
\end{tabular}
\end{center}
\vskip -0.3in
\end{table}

\subsection{Convergence Performance}\label{sub-sec:conv}

We next conduct the experiments to study the convergence performance of PPM for recovering the communities in graphs generated by the symmetric SBM in Definition \ref{SBM}. In the simulations, we choose three different sets of $(\alpha,\beta, K)$ such that $\sqrt{\alpha} - \sqrt{\beta} > \sqrt{K}$ and generate graphs of dimension $n=6000$. Moreover, we generate the initial point $\bH^0$ in Algorithm \ref{alg:PGD} via $\bH^0 \in \mT(\bG)$, where each entry of $\bG \in \R^{n\times K}$ is randomly generated by the standard normal distribution. Let $\bH^k$ denote the $k$-th iterate of the PPM. In each graph, we run PPM $10$ times from different initial points and then plot the distances of the iterates to the ground truth, i.e., $\min_{\bQ \in \Pi_K}\|\bH^k-\bH^*\bQ\|_F$, against the iteration number in Figure \ref{fig-3}. It can be observed that PPM exhibits a finite termination phenomenon and converges to the ground truth within $20$ iterations even if it starts from a randomly generated initial point. This also corroborates the one-step convergence result in Lemma \ref{lem:one-step-conv} and the iteration complexity in Theorem \ref{thm-1}. 

\subsection{Recovery Efficiency and Accuracy}

Finally, we conduct the experiments to compare the recovery efficiency and accuracy of our method with SDP, SC, and PMLE on real data sets. We use the data sets \emph{polbooks}, \emph{polblogs}, and \emph{football} downloaded from the SuiteSparse Matrix Collection \citep{davis2011university}.\footnote{\url{https://sparse.tamu.edu/}} For the set \emph{football}, we remove the communities whose sizes are less than 10. To tackle the difficulty that these real networks have unbalanced communities, we modify the second constraint in \eqref{set-H} as $\bH^T\bo_K=\bm{\pi}$, where $\pi_k$ denotes the $k$-th community size for all $k\in [K]$, and then apply PPM for solving the resulting formulation as in Algorithm \ref{alg:PGD}. The stopping criteria for the tested methods are set as follows. For PPM, we terminate it when there exists some iterate $k\ge 6$ such that  $\|\bH^k-\bH^{l}\|_F \le 10^{-3}$ for some $k-5\le l \le k-1$; for ADMM, we terminate it when the norm of difference of two consecutive iterates is less than $10^{-3}$. No stopping criterion is needed for SC and PMLE since SC employs the MATLAB function \textsf{kmeans} to do the clustering and PMLE directly assigns each vertex to the corresponding community based on the initialization partition. Besides, we generate an initial point for PPM as in Section \ref{sub-sec:conv}. Then, we run each algorithm $10$ times and select the best solution (in terms of function value) as its recovery solution. Moreover, we set the maximum iteration number for PPM and ADMM as 1000. To compare the recovery efficiency and accuracy of the tested methods, we report the total CPU time for all runs and the number of misclassified vertices (MVs) of each method in Table \ref{table-2}. These results, together with those in Table \ref{table-1}, demonstrate that our proposed method is comparable to these state-of-the-art methods in terms of recovery efficiency and accuracy on both synthetic and real data sets.  

\begin{table}[!htbp]
\caption{Total CPU times (in seconds) and the number of misclassified vertices (MVs) of the methods on real data sets.}
\label{table-2}
\begin{center}
\begin{tabular}{lcccc}
\toprule
 Time (s) & PPM & SDP & SC & PMLE \\
\midrule
\emph{polbooks} & {\bf 0.28} & 10.26 & 0.30 & 19.67\\ 
\emph{polblogs} & {\bf 0.02} & 2348 & 0.41 & 1.39\\
\emph{football} & {\bf 0.21} & 0.83 & 0.42 & 0.40 \\
\midrule
num. of MVs &  PPM & SDP & SC & PMLE \\
\midrule
\emph{polbooks} & {\bf 18} & 24 & {\bf 18} & 19\\ 
\emph{polblogs} & {\bf 52} & 238 & 215 & 279\\
\emph{football} & 4 & {\bf 2} & {\bf 2} & 13\\
\bottomrule
\end{tabular}
\end{center}
\end{table}

\section{Concluding Remarks}\label{sec:con}
% that starts from an initial point satisfying a partial recovery condition % with an initialization satisfying a partial recovery condition % with a qualified initialization 
In this work, we proposed a projected power method for solving the ML formulation of the symmetric SBM. We showed that provided an initial point satisfying a mild partial recovery condition, this method achieves exact recovery down to the information-theoretic threshold and runs in $\mO(n\log n/\log\log n)$ time in the logarithmic degree regime. This is also demonstrated by our numerical results. Moreover, it is observed in the numerical results that the proposed method still works effectively even with a random initialization. % Our numerical results show that our method achieves exact recovery at the information-theoretic limit and has a good convergence performance even with a random initialization.
 Then, one natural future direction is to study the convergence behavior of the proposed method with a random initialization. Another direction is to extend our proposed method to other variants of the basic SBM, such as degree-corrected block models (see, e.g., \citet{gao2018community, karrer2011stochastic}), labelled SBMs (see, e.g., \citet{heimlicher2012community, yun2016optimal}), and overlapping SBMs (see, e.g., \citet{airoldi2008mixed,gopalan2013efficient}).
% nearly at the information-theoretic limit up to a factor of $\sqrt{2}$ % develop techniques to improve our theoretical results to the optimal threshold. Another direction is 

\section*{Acknowledgements}

This work is supported in part by CUHK Research Sustainability of Major RGC Funding Schemes project 3133236. % The first and forth authors are supported in part by the Hong Kong Research Grants Council (RGC) General Research Fund (GRF) project CUHK 14208117 and in part by the CUHK Research Sustainability of Major RGC Funding Schemes project 3133236. % The second author is supported in part by the National Natural Science Foundation of China (NSFC) project 11901490 and in part by an HKBU Start-up Grant. Most of the work of the second author was done when he was affiliated with the Department of Mathematics of Hong Kong Baptist University.

\bibliographystyle{icml2021}

% \bibliography{C:/Users/wangpeng/Dropbox/Submitted-Reviewed-Papers/Submitted-Conference-Papers/ICML-2021/K-SBM}

\newpage
\onecolumn
\vspace{0.1cm}
\begin{center}
{\Large \bf Supplementary Material}
\end{center}
\vspace{-0.3cm}
\par\noindent\rule{\textwidth}{1pt}
\setcounter{section}{0}

\renewcommand\thesection{\Alph{section}}
\renewcommand\thesubsection{\arabic{subsection}}

In the appendix, we provide proofs of some technical results presented in Sections \ref{sec:preli} and \ref{sec:pf-main}. To proceed, we introduce some further notations. Given two random variables $X$ and $Y$, we write $X\overset{d}{=}Y$ if $X$ and $Y$ are equal in distribution. We use $\ve(\bH) \in \R^{nK}$ to denote the vectorization of $\bH \in \R^{n\times K}$  formed by stacking its columns into a single column vector. We use $\be_i$ to denote a standard basis with a 1 in the $i$-th coordinate and $0$'s elsewhere. We use $\otimes$ to denote the Kronecker product. We use $\bo_n$ and $\bE_n$ to denote the $n$-dimensional all-one vector and $n\times n$ all-one matrix, respectively, and simply write $\bo$ and $\bE$ when their dimension can be inferred from the context.

\section{Proofs in Section \ref{sec:preli}} 

Before we proceed, let us introduce the definition of the minimum-cost assignment problem (MCAP) formally; see \citet[Definitions 1.1, 1.2]{tokuyama1995geometric}.
\begin{defi}\label{def:MCAP}
Let $\Gamma=(U,V,E)$ be a complete bipartite graph, where $U$ with $n$ nodes $u_1,\dots,u_n$ and $V$ with $K$ nodes $v_1,\dots,v_K$ denotes two parts of the graph and $E$ denotes the edges of the graph. For a cost matrix $\bC \in \R^{n\times K}$, each entry $c_{ik}$ is the cost associated with the edge $e(u_i,v_k) \in E$. Given a vector $\bm{\pi}=(\pi_1,\dots,\pi_K)$ such that each $\pi_k$ is a non-negative integer and $\sum_{k=1}^K\pi_k=n$, the minimum-cost $\bpi$-assignment problem is to find a subgraph of $\Gamma$ denoted by $G$ such that (i) the node set of $G$ is $U \cup V$, the degree of each node $u_i \in U$ is 1, and the degree of each node $v_k \in V$ is $\pi_k$, and (ii) the total cost $\sum_{e(u_i,v_k)\in E(G)} c_{ik}$ is minimized, where $E(G)$ denotes the edges of graph $G$.  % $G=(U,V,E^\prime)$, $\deg(u_i)=1$ for all $u_i \in U$, and $\deg(v_k)=\pi_k$ for all $v_k\in V$, and the total cost $\sum_{(u_i,v_k)\in G} c_{ik}$ is minimized. 
\end{defi}

\subsection{Proof of Proposition \ref{prop:MCAP}}

\begin{proof}
According to Definition \ref{def:MCAP}, for any $\bC \in \R^{n\times K}$, Problem \eqref{project-H} is equivalent to a minimum-cost $\bpi$-assignment problem with the cost matrix being $-\bC$ and $\bpi=m\bo_K$. According to \citet[Theorem 2.1, Proposition 3.4]{tokuyama1995geometric}, this problem can be solved in $\mO(K^2n\log n)$ time. 
\end{proof}

\subsection{Proof of Corollary \ref{coro:time}}
\begin{proof}
First, we derive the time complexity of computing the matrix multiplication of $\bA\bH$. Let $\ba \in \R^n$ denote a column of $\bA$. Since $\bA$ is generated according to the symmetric SBM with $p=\alpha\log n/n$ and $q=\beta\log n/n$, then we have
\begin{align}\label{eq1:pf-coro-time}
\|\ba\|_0 \overset{d}{=} \sum_{i=1}^m W_i + \sum_{i=1}^{n-m}Z_i,
\end{align}
where $\{W_i\}_{i=1}^m$ are i.i.d. $\mathbf{Bern}(p)$ and $\{Z_i\}_{i=1}^{n-m}$ are i.i.d. $\mathbf{Bern}(q)$, independent of $\{W_i\}_{i=1}^m$. It then follows that
\begin{align*}
\E[\|\ba\|_0] = mp+(n-m)q,\qquad \mathbf{Var}[\|\ba\|_0] = mp(1-p) + (n-m)q(1-q) \le mp + (n-m)q. 
\end{align*}
Applying the Bernstein's inequality to the bounded distribution in \eqref{eq1:pf-coro-time} yields that
\begin{align*}
\P\left( \left| \|\ba\|_0 - (mp+(n-m)q) \right| \ge 2(mp+(n-m)q) \right) &\le 2\exp\left( -\frac{4(mp+(n-m)q)^2/2}{mp + (n-m)q + 2(mp + (n-m)q)/3} \right) \\
& \le  2\exp\left( -(mp+(n-m)q) \right) \\
% & = 2\exp\left( -\frac{\alpha+(K-1)\beta}{K}\log n \right) \\
& = 2n^{-\frac{\alpha+(K-1)\beta}{K}}.
\end{align*}
This implies 
\begin{align*}
\P\left( \|\ba\|_0 < 3(mp+(n-m)q) \right) \ge 1 - 2n^{-\frac{\alpha+(K-1)\beta}{K}}.
\end{align*}
Upon applying the union bound to the $n$ columns of $\bA$, we conclude that it holds with probability at least $1-2n^{1-\frac{\alpha+(K-1)\beta}{K}}$ that the number of non-zero entries in $\bA$ is less than $3n(mp+(n-m)q)=\frac{3\alpha+3(K-1)\beta}{K}n\log n$. Thus, the time complexity of computing the matrix multiplication of $\bA\bH$ is $3(\alpha+(K-1)\beta)n\log n$ with probability at least $1-2n^{1-\frac{\alpha+(K-1)\beta}{K}}$. Besides, since $\sqrt{\alpha}-\sqrt{\beta} > \sqrt{K}$, then $1-\frac{\alpha+(K-1)\beta}{K} < 0$. These, together with Proposition \ref{prop:MCAP}, Theorem \ref{thm-1}, and the union bound, imply the desired result. 
\end{proof}

\section{Proofs in Section \ref{sub-sec:pgd}} 

\subsection{Proof of Lemma \ref{lem:form-H}}  

To proceed, let us formally introduce the definition of total unimodularity.
\begin{defi}\label{def:TU}
A matrix $\bA$ is totally unimodular if $\det(\bB) \in \{0,1,-1\}$ for every square non-singular submatrix $\bB$ of $\bA$.  
\end{defi}

\begin{proof}
The equivalence between (i) and (iii) is obvious. Next, suppose that (ii) holds. By letting $\bx = \ve(\bH)$, $\bH$ is an extreme point of $\mP$ if and only if $\bx$ is an extreme point of $\mP^\prime := \{\bx\in \R^{nK}:\bA\bx = \bb,\ \bx \ge \b0\}$, where 
\begin{align*}
\bA = \begin{bmatrix}
\bA_1 \\
\bA_2
\end{bmatrix},\ \bA_1 = \begin{bmatrix}
\be_1^T & \be_1^T & \cdots & \be_1^T \\
\be_2^T & \be_2^T & \cdots & \be_2^T \\
\vdots & \vdots & \ddots & \vdots \\
\be_n^T & \be_n^T & \cdots & \be_n^T
\end{bmatrix},\ \bA_2 = \begin{bmatrix}
\bo_n^T & \b0 & \dots & \b0 \\
\b0 & \bo_n^T  & \dots & \b0 \\ 
\vdots & \vdots & \ddots & \vdots \\
\b0 & \b0 & \dots & \bo_n^T
\end{bmatrix}, \text{and}\ \bb = \begin{bmatrix}
\bo_n \\ m\bo_{K}
\end{bmatrix}.
\end{align*}
Let $\ba_i^T$ denote the $i$-th row of $\bA$ for $i=1,\dots,n+K$, $\bA^\prime$ denote the submatrix of $\bA$ obtained by removing its $(n+K)$-th row, and $\bm{b}^\prime$ denote the subvector of $\bb$ obtained by removing its $(n+K)$-th element. Then, one can verify that all rows of $\bA^\prime$ are linearly independent and the rank of $\bA^\prime$ being $n+K-1$. This implies that $\ba_{n+K}^T\bx=m$ is a redundant constraint in the above linear system and it can be omitted. Consequently, we have  $\mP^\prime = \{\bx\in \R^{nK}:\bA^\prime\bx = \bb^\prime,\ \bx \ge \b0\}$. According to \citet[Theorem 3]{hoffman2010integral}, $\bA^\prime$ is totally unimodular. Then, let $\bB$ be any basis of $\bA^\prime$, which is essentially a subset of the columns of $\bA^\prime$ of rank $n+K-1$. Since $\bA^\prime$ is totally unimodular and $\bB$ is a square non-singular submatrix of $\bA^\prime$, $\bB$ is unimodular by Definition \ref{def:TU}. This implies that every basis of $\bA^\prime$ is unimodular. This, together with $\bb^\prime$ is an integer vector and the theorem in \citet{veinott1967integral}, implies that any extreme point $\bx$ of $\mP^\prime$ satisfies $\bA^\prime\bx=\bb^\prime,\ \bx \in \{0,1\}^{nK}$. Hence, $\bH \in \mH$. 

Now, suppose that (iii) holds. We show that (ii) holds. Suppose to the contrary that $\bH$ is not an extreme point of $\mP$. Then, there exist $\bH^1,\bH^2 \in \mP$ and $\bH^1 \neq \bH^2$ such that $\bH = (\bH^1+\bH^2)/2$. Besides, there exist indices $i,j$ such that $h_{ij}^1\neq h_{ij}^2$. This, together with $h_{ij}^1,h_{ij}^2 \in [0,1]$ and $\bH = (\bH^1+\bH^2)/2$, yields that $h_{ij}=(h_{ij}^1 + h_{ij}^2)/2 \in (0,1)$, which contradicts the form of $\bH$ in \eqref{form-H}. Hence, $\bH$ is an extreme point of $\mP$. 

As a result, (iii) $\Rightarrow$ (ii) $\Rightarrow$ (i) $\Rightarrow$ (iii) and thus the proof is completed. 
\end{proof}

\subsection{Proof of Proposition \ref{prop:LP}}  

\begin{proof}
According to \eqref{project-H}, we have
\begin{align*}
\mT(\bC) = \argmax\left\{ \langle \bC,\bH \rangle: \ \bH \in \mH \right\}  =  \argmax\left\{ \langle \bC,\bH \rangle: \ \bH \in \mP \right\},
\end{align*}
where the first equality is due to $\|\bH\|_F=\sqrt{n}$ for all $\bH \in \mH$ and the second equality follows from Lemma \ref{lem:form-H} and the fact that there exists a vertex (i.e., extreme point) of $\mP$ that is optimal for the LP in \eqref{LP-H}. Then, the proof is completed.  
\end{proof}

\subsection{Proof of Lemma \ref{lem:closed-form-LP}} 

\begin{proof}
Let us consider the KKT system of the LP in \eqref{LP-H}, i.e.,
\begin{align}\label{KKT:LP}
\left\{  
\begin{aligned}
&  -\bC+\bu\bo_K^T +\bo_n\bw^T = \bLam,\ \bLam \ge \b0,\\ 
&  \langle \bLam, \bH \rangle = 0,  \\
\end{aligned} 
\right.
\end{align}  
where $\bu\in \R^n$, $\bw \in \R^K$, and $\bLam \in \R^{n\times K}$ are the dual variables associated with the constraints $\bH\bo_K=\bo_n$, $\bH^T\bo_n=m\bo_K$, and $\bH \ge 0$, respectively. According to Proposition \ref{prop:LP} and Lemma \ref{lem:form-H}, the optimal solutions of the LP in \eqref{LP-H} take the form of \eqref{form-H}. For $i$ such that $i\in \mI_k$ and $k \in [K]$, since $h_{ik}=1$, then $\lambda_{ik}=0$ and 
\begin{align}\label{eq-1:pro-closed-form-LP}
c_{ik} = u_i + w_k. 
\end{align}
Besides, for $j$ such that $j \notin \mI_k$ and $k \in [K]$, since $h_{jk}=0$, we have $ \lambda_{jk} \ge 0$ and $-c_{jk} + u_j + w_k = \lambda_{jk}$. This implies
\begin{align}\label{eq-2:pro-closed-form-LP}
c_{jk} \le u_j + w_k. 
\end{align}
For any $i\in \mI_k$ and $j\in\mI_\ell$ with $k\neq \ell$, we have $c_{ik}=u_i+w_k$, $c_{jk}\le u_j+w_k$, $c_{j\ell}=u_j+w_\ell$ and $c_{i\ell} \le u_i + w_\ell$ due to \eqref{eq-1:pro-closed-form-LP}, \eqref{eq-2:pro-closed-form-LP}, and $\mI_k \cap \mI_\ell = \emptyset$. 
% For any $i \in \mI_k$, \eqref{eq-1:pro-closed-form-LP} is equivalent to $u_i=c_{ik}-w_k$. 
This implies
\begin{align}\label{eq-0:lem-LP-cont}
c_{ik} - c_{i\ell}  \ge w_k - w_\ell \ge  c_{jk} - c_{j\ell},\ \forall\ i\in \mI_k,j\in \mI_\ell, 1 \le k\neq \ell \le K.
\end{align}
Conversely, suppose that there exists $\bw\in \R^K$ such that \eqref{eq-0:lem-LP-cont} holds. By letting $u_i=c_{ik}-w_k$ for any $i\in \mI_k$ and $k\in [K]$, then we have $c_{ik}=u_i+w_k$ for any $i\in \mI_k$ and $c_{j\ell}=u_j+w_\ell$ for any $j \in \mI_\ell$, where $k\neq \ell$. This, together with \eqref{eq-0:lem-LP-cont}, implies \eqref{eq-1:pro-closed-form-LP} and \eqref{eq-2:pro-closed-form-LP}. Hence, the proof is completed.  
\end{proof}

In addition, we have another interesting result for the LP in \eqref{LP-H}, which will be used in the later proofs. 
\begin{lemma}\label{lem:LP-Q}
For a matrix $\bC \in \R^{n\times K}$, $\bH \in \mT(\bC)$ if and only if $\bH\bQ \in \mT(\bC\bQ)$ for some $\bQ \in \Pi_K$. 
\end{lemma}
\begin{proof}
Suppose that $\bH \in \mT(\bC)$. For any $\bG \in \mH$, we have
\begin{align*}
\langle \bC\bQ, \bG \rangle = \langle \bC, \bG\bQ^T \rangle \le  \langle \bC, \bH \rangle, 
\end{align*}
where the inequality is due to $\bH \in \mT(\bC)$ and $\bG\bQ^T \in \mH$ for a $\bQ \in \Pi_K$. Moreover, $\langle \bC\bQ, \bH\bQ \rangle=\langle \bC, \bH \rangle$ and $\bH\bQ \in \mH$, and thus $\bH\bQ \in \mT(\bC\bQ)$. Suppose that $\bH\bQ \in \mT(\bC\bQ)$ for a $\bQ \in \Pi_K$.
% Then, for any $\bG \in \mH$, 
%\begin{align*}
%\langle \bC, \bG \rangle = \langle \bC\bQ, \bG\bQ \rangle \le \langle \bC\bQ, \bH\bQ \rangle = \langle \bC, \bH \rangle.
%\end{align*}
%where the inequality becomes equality if $\bG = \bH$. This, together with $\bH \in \mH$, implies $\bH \in \mT(\bC)$.
By the same argument as above, we have $\bH \in \mT(\bC)$. Thus, the proof is completed. 
\end{proof}

\subsection{Proof of Lemma \ref{lem:LP-cont}} 

\begin{proof}
Note that \eqref{eq:LP-cont-cond} implies that for all $1\le k \neq \ell \le K$, $i \in \mI_k$, and $j\in \mI_\ell$,
\begin{align*}
c_{ik} - c_{i\ell} > 0 > c_{jk} - c_{j\ell}.
\end{align*}
This, together with Lemma \ref{lem:closed-form-LP}, yields that $\mT(\bC)$ is a singleton and $\{\bV\} = \mT(\bC)$ satisfies for all $k\in [K]$,
\begin{align}\label{eq-1:lem-LP-cont}
v_{ik} = 
\begin{cases}
1,\ \text{if}\ i \in \mI_k, \\
0,\ \text{otherwise}.
\end{cases}
\end{align}
Let $\bC^\prime \in \R^{n\times K}$ be arbitrary and $\bVp \in \mT(\bCp)$. It then follows from Lemma \ref{lem:closed-form-LP} that
\begin{align}\label{eq-2:lem-LP-cont}
v^\prime_{ik} = 
\begin{cases}
1,\ \text{if}\ i \in \mJ_k, \\
0,\ \text{otherwise},
\end{cases}
\end{align}
where $\mJ_1,\dots,\mJ_K$ satisfy ${\cup}_{k=1}^K\mJ_k = [n]$, $\mJ_k \cap \mJ_\ell = \emptyset$, and $|\mJ_k| = m$, and there exists $\bw^\prime \in \R^K$ such that
\begin{align}\label{eq-3:lem-LP-cont}
c^\prime_{ik}-c^\prime_{i\ell} \ge w_k^\prime - w_\ell^\prime \ge  c_{jk}^\prime - c_{j\ell}^\prime,\ \forall\ i\in \mJ_k,j\in \mJ_\ell, 1 \le k\neq \ell \le K.
\end{align}
For ease of exposition, let $\mI^c_k = [n]\setminus \mI_k = \cup_{\ell\neq k} \mI_\ell$ and $\mJ^c_k = [n]\setminus \mJ_k = \cup_{\ell\neq k} \mJ_\ell$. Since 
\begin{align}\label{eq-7:lem-LP-cont}
|\mI_k \cap \mJ_k| + |\mI_k \cap \mJ_k^c| = |\mI_k| =  m,\ |\mI_k \cap \mJ_k| + |\mI_k^c \cap \mJ_k| = |\mJ_k| =  m, 
\end{align}
we deduce that $|\mI_k \cap \mJ_k^c|=|\mI_k^c \cap \mJ_k|=s_k$ for some $0 \le s_k \le m$ for all $k\in [K]$. By \eqref{eq-1:lem-LP-cont} and \eqref{eq-2:lem-LP-cont}, we have for all $k\in[K]$,
\begin{align*}
v_{ik} - v^\prime_{ik} = \begin{cases}
0,\ i \in (\mI_k \cap \mJ_k) \cup\ (\mI_k^c\cap \mJ_k^c), \\
1,\ i \in \mI_k \cap \mJ_k^c, \\
-1,\ i \in \mI_k^c \cap \mJ_k.
\end{cases}
\end{align*} 
Since $|\mI_k \cap \mJ_k^c|=|\mI_k^c \cap \mJ_k|=s_k$, this yields
\begin{align}\label{eq-6:lem-LP-cont}
\|\bV-\bVp\|_F^2 = 2\sum_{k=1}^K s_k.
\end{align}
On the other hand, for any $i\in \mI_k \cap \mJ_\ell$ and $k\neq \ell$, by letting $x_i^{k\ell} = c_{ik}-c_{i\ell}+c_{i\ell}^\prime-c_{ik}^\prime-(w_\ell^\prime-w_k^\prime)$, we have
\begin{align}\label{eq-4:lem-LP-cont}
\left(c_{ik}-c_{ik}^\prime\right)^2 + \left(c_{i\ell}-c_{i\ell}^\prime\right)^2  \ge \frac{1}{2}(x_i^{k\ell}+w_\ell^\prime - w_k^\prime)^2,
\end{align}
where the inequality is due to $a^2+b^2 \ge (a+b)^2/2$ for any $a,b\in\R$. According to \eqref{eq:LP-cont-cond} and \eqref{eq-3:lem-LP-cont}, we have for any $i\in \mI_k \cap \mJ_\ell$ and $k\neq \ell$,
\begin{align}\label{eq-5:lem-LP-cont}
x_i^{k\ell} \ge \delta. 
\end{align}
Then, consider
\begin{align}\label{eq-8:lem-LP-cont}
\|\bC-\bCp\|_F^2 & = \sum_{i=1}^n\sum_{j=1}^K  (c_{ij}-c_{ij}^\prime)^2 \ge \sum_{j=1}^K \sum_{k=1}^K \sum_{i\in \mI_k \cap \mJ_k^c}(c_{ij}-c_{ij}^\prime)^2 = \sum_{j=1}^K \sum_{k=1}^K \sum_{\ell \neq k}^K  \sum_{i\in \mI_k \cap \mJ_\ell}(c_{ij}-c_{ij}^\prime)^2 \notag \\
& \ge  \sum_{k=1}^K \sum_{\ell \neq k}^K  \sum_{i\in \mI_k \cap \mJ_\ell} \left( \left(c_{ik}-c_{ik}^\prime\right)^2 + \left(c_{i\ell}-c_{i\ell}^\prime\right)^2  \right) \notag\\
& \ge \sum_{k=1}^K \sum_{\ell \neq k}^K  \sum_{i\in \mI_k \cap \mJ_\ell} \frac{1}{2} \left(x_i^{k\ell} + w_\ell^\prime - w_k^\prime \right)^2,
\end{align}
where the last inequality is due to \eqref{eq-4:lem-LP-cont} and note that $x_i^{k\ell} \ge \delta$ for any $i\in \mI_k \cap \mJ_\ell$ and $k\neq \ell$ by \eqref{eq-5:lem-LP-cont}. Then, we consider the following optimization problem:
\begin{align*}
\min_{\bm{x},\bm{w}}\ & f(\bm{x},\bw) := \sum_{k=1}^K \sum_{\ell \neq k}^K  \sum_{i\in \mI_k \cap \mJ_\ell} \frac{1}{2} \left(x_i^{k\ell} + w_\ell - w_k \right)^2 \\
\st\ & \quad x_i^{k\ell} \ge \delta,\ \forall\ i\in \mI_k \cap \mJ_\ell,\ k \neq \ell.
\end{align*} 
We claim that the optimal solution of this problem is $w_1^*=\dots=w_K^*$ and $(x_i^{k\ell})^*=\delta$ for all $i\in \mI_k \cap \mJ_\ell$ and $k\neq \ell$. Indeed, the KKT system of the above problem is 
\begin{align*}% \label{KKT:LP}
\left\{  
\begin{aligned}
&  x_i^{k\ell}+w_\ell - w_k - \lambda_i^{k\ell} = 0,\ \forall\ i\in \mI_k \cap \mJ_\ell,\ k \neq \ell, \\ 
&   \sum_{\ell \neq j}^K \sum_{i \in \mI_j \cap \mJ_\ell} (w_j-w_\ell-x_i^{j\ell}) +  \sum_{k \neq j}^K \sum_{i \in \mI_k \cap \mJ_j} (w_j-w_k+x_i^{j\ell}) = 0,\ \forall j \in [K],  \\
& (x_i^{k\ell} - \delta) \lambda_i^{k\ell} = 0,\ \lambda_i^{k\ell} \ge 0,\ \forall\ i\in \mI_k \cap \mJ_\ell,\ k \neq \ell,
\end{aligned} 
\right.
\end{align*}  
where $ \lambda_i^{k\ell} \in \R$ is the dual variable associated with the constraint $\delta - x_i^{k\ell} \le 0$ for any $i\in \mI_k \cap \mJ_\ell$ and  $k \neq \ell$. Then, one can verify that $w_1^*=\dots=w_K^*$ and $(x_i^{k\ell})^*=\delta$ for all $i\in \mI_k \cap \mJ_\ell$ and $k\neq \ell$ satisfy this KKT system. According to \eqref{eq-8:lem-LP-cont}, we further have
\begin{align*}
\|\bC-\bCp\|_F^2 \ge f(\bm{x},\bm{w}) \ge f(\bm{x}^*,\bw^*) = \frac{1}{2}\delta^2\sum_{k=1}^Ks_k.  
\end{align*}
This, together with \eqref{eq-6:lem-LP-cont}, implies the desired result in \eqref{rst:LP-cont}. 
\end{proof} 

\subsection{Proof of Lemma \ref{lem:contra}} 

Without loss of generality, we assume that $\bHs = \bI_K \otimes \bo_m$ in Definition \ref{SBM}. Since $\bA$ is generated according to the symmetric SBM in Definition \ref{SBM}, one can verify
\begin{align}\label{EA}
\E[\bA] = \bB \otimes \bE_m = \frac{p+(K-1)q}{K}\bE_n + (p-q)\bU\bU^T  \otimes \bE_m,
\end{align}
where 
\begin{align}\label{EA-U}
\bB = \begin{bmatrix}
p & q & \dots & q\\
q & p & \dots & q\\
\vdots & \vdots & \ddots & \vdots\\
q & q & \dots & p
\end{bmatrix} \in \R^{K\times K}\ \text{and}\ \bU = \begin{bmatrix}
\frac{1}{\sqrt{2}} & \frac{1}{\sqrt{6}} & \dots & \frac{1}{\sqrt{(K-1)K}}   \\
-\frac{1}{\sqrt{2}} & \frac{1}{\sqrt{6}} & \dots & \frac{1}{\sqrt{(K-1)K}} \\
0 & -\frac{\sqrt{2}}{\sqrt{3}} & \dots & \frac{1}{\sqrt{(K-1)K}} \\
\vdots & \vdots &  \ddots & \vdots \\
0 & 0 & \dots & -\frac{\sqrt{K-1}}{\sqrt{K}}
\end{bmatrix} \in \R^{K\times (K-1)}. 
\end{align}
Moreover, one can verify
\begin{align}\label{U^TU}
\bU^T\bU = \bI_{K-1}.
\end{align}
\begin{proof}
Let us decompose $\bH$ into two parts that are orthogonal:
\begin{align*}
\bH = \bHs\bQ\bZ + \bG,\ \text{where}\ \bG^T\bHs = \b0.
\end{align*}
Then, one can verify $\bZ = (\bH^{*}\bQ)^T\bH/m$. This, together with $\bH, \bHs \in \mH$, implies $\bZ\bo_K=\bo_K$ and $z_{k\ell} \in [0,1]$ for all $k,\ell \in [K]$. Using the mixed-product property of the Kronecker product, we have
\begin{align}\label{eq-7:lem-con}
(\bU\otimes \bo_m)^T\bHs = (\bU\otimes \bo_m)^T(\bI_K \otimes \bo_m) = m\bU^T.
\end{align}
Note that $ \bG^T\bHs = \b0$ with $\bHs=\bI_K\otimes\bo_m$, and thus we have $(\bU\otimes \bo_m)^T\bG=\b0$. This, together with \eqref{eq-7:lem-con}, yields
\begin{align}\label{eq-8:lem-con}
(\bU\otimes \bo_m)^T\bH = (\bU\otimes \bo_m)^T(\bHs\bQ\bZ+\bG) = m\bU^T\bQ\bZ.
\end{align}
According to \eqref{EA}, we have
\begin{align}\label{eq-9:lem-con}
\E[\bA](\bH-\bHs\bQ) & =  \frac{p+(K-1)q}{K}\bE_n(\bH-\bHs\bQ) + (p-q)(\bU\bU^T \otimes \bE_m )(\bH-\bHs\bQ) \notag\\
& = (p-q)(\bU\bU^T \otimes \bE_m )(\bH-\bHs\bQ) \notag\\
& = (p-q)(\bU\otimes \bo_m)(\bU \otimes \bo_m)^T(\bH-\bHs\bQ) \notag \\
& = m(p-q)(\bU\otimes \bo_m)\bU^T\bQ(\bZ-\bI),
\end{align}
where the second equality is due to $\bE_n\bH=\bE_n\bHs = m\bE_{n,K}$, the third equality is because of the mixed-product property of the Kronecker product, and the last equality follows from \eqref{eq-7:lem-con} and \eqref{eq-8:lem-con}. Suppose that the following inequality holds: 
\begin{align}\label{eq-10:lem-con}
m\|\bI-\bZ\|_F \le 4\varepsilon\sqrt{n}\|\bH-\bHs\bQ\|_F.
\end{align}
This immediately implies the desired result, because
\begin{align*}
\|\bA\bH-\bA\bHs\bQ\|_F & = \left\| \E[\bA](\bH-\bHs\bQ) + \Delta(\bH-\bHs\bQ) \right\|_F\\  
& =  \left\|m(p-q)(\bU\otimes \bo_m)\bU^T\bQ(\bZ-\bI) +  \Delta(\bH-\bHs\bQ) \right\|_F \\
%& \le (p-q)\|\bU\otimes \bo_m\| \| (\bU\otimes \bo_m)^T(\bH-\bHs)\|_F + \|\Delta\|\|\bH-\bHs\|_F \\
%& \le \sqrt{m}(p+(K-1)q) \cdot m\|\bI-\bZ\|_F + \|\Delta\|\|\bH-\bHs\bQ\|_F \\
& \le \sqrt{m}(p-q)\cdot m\|\bI-\bZ\|_F + \|\Delta\|\|\bH-\bHs\bQ\|_F \\
& \le \left(\frac{4\varepsilon n}{\sqrt{K}}(p-q) + \|\Delta\| \right) \|\bH-\bHs\bQ\|_F,
\end{align*}
where the second equality is due to \eqref{eq-9:lem-con}, the first inequality follows from the triangle inequality, $\|\bU \otimes \bo_m\|=\sqrt{m}$, and $\|(\bU^T\bQ)^T\bU^T\bQ\|=\|\bU\bU^T\|=1$, and the second inequality is because of \eqref{eq-10:lem-con}.

The rest of the proof is devoted to proving \eqref{eq-10:lem-con}. We can verify
\begin{align}\label{eq-1:lem-con}
\|\bI-\bZ\|_F^2 = \sum_{k=1}^K (1-z_{kk})^2 + \sum_{k\neq \ell} z_{k\ell}^2. 
\end{align}
% This, together with $z_{k\ell} \in [0,1] $ for all $k,\ell \in [K]$, yields
Besides, we have
\begin{align}\label{eq-2:lem-con}
\|\bI-\bZ\|_F  \le \sum_{k=1}^K |1-z_{kk}| + \sum_{k\neq \ell} |z_{k\ell}|  =  \sum_{k=1}^K (1-z_{kk}) + \sum_{k\neq \ell} z_{k\ell} = 2 \sum_{k=1}^K (1-z_{kk}),
\end{align}
where the first equality follows from $z_{k\ell} \in [0,1] $ for all $k,\ell \in [K]$ and the second equality is due to $\bZ\bo=\bo$. Note that $\bH^T\bH = m\bI$ due to $\bH \in \mH$, which is equivalent to $m\bZ^T\bZ + \bG^T\bG = m\bI$. This implies 
\begin{align}\label{eq-3:lem-con}
\|\bG\|_F^2 = mK - m\sum_{k=1}^K\sum_{\ell=1}^K z_{k\ell}^2.
\end{align}
According to $\|\bH-\bHs\bQ\|_F \le \varepsilon \sqrt{n}$, we obtain
\begin{align}\label{eq-4:lem-con}
\|\bH-\bHs\bQ\|_F^2 = m\|\bZ-\bI\|_F^2 + \|\bG\|_F^2 \le \varepsilon^2n.  
\end{align}
This, together with \eqref{eq-1:lem-con} and \eqref{eq-3:lem-con}, implies
\begin{align}\label{eq-5:lem-con}
\sum_{k=1}^Kz_{kk} \ge \left(1-\frac{\varepsilon^2}{2}\right) K.
\end{align}
Then, for any $\ell \in [K]$, we have
\begin{align}\label{eq-6:lem-con}
z_{\ell\ell} \ge  \left(1-\frac{\varepsilon^2}{2}\right) K - \sum_{k\neq \ell} z_{kk} \ge 1 - \frac{K}{2}\varepsilon^2 \ge \frac{1}{2},
\end{align}  
where the first inequality is due to \eqref{eq-5:lem-con}, the second inequality is because of $z_{kk} \le 1 $ for all $k \in [K]$, and the last inequality uses $\varepsilon \in (0,1/\sqrt{K})$.
According to \eqref{eq-3:lem-con}, we have
\begin{align*}
\frac{\|\bG\|_F^2}{m} = K - \sum_{k=1}^{K}z_{kk}^2 - \sum_{k\neq \ell}z_{k\ell}^2 \ge K - \sum_{k=1}^{K}z_{kk}^2 - \sum_{k\neq \ell}z_{k\ell} = \sum_{k=1}^K z_{kk}(1-z_{kk}) \ge \frac{1}{2}\sum_{k=1}^K (1-z_{kk}),
\end{align*}
where the first inequality is due to $z_{k\ell} \in [0,1]$ for all $k,\ell \in [K]$, the second equality is because of $\bZ\bo=\bo$, and the second inequality uses \eqref{eq-6:lem-con} and $z_{kk} \in [0,1]$ for all $k\in [K]$. This, together with \eqref{eq-2:lem-con}, yields 
\begin{align*}
\|\bI-\bZ\|_F \le \frac{4\|\bG\|_F^2}{m} \le \frac{4}{m}\|\bH-\bHs\bQ\|_F^2 \le \frac{4\varepsilon\sqrt{n}}{m}\|\bH-\bHs\bQ\|_F,
\end{align*}
where the second and third inequalities are due to \eqref{eq-4:lem-con}.

\end{proof}

\subsection{Proof of Lemma \ref{lem:block-gap}}

\begin{proof}
Since $\sqrt{\alpha}-\sqrt{\beta} > \sqrt{K}$, there exists a constant $\gamma> 0$, whose value only depends on $\alpha$, $\beta$, and $K$, such that
\begin{align}\label{eq-1:lem-block-gap}
c_2:=\frac{(\sqrt{\alpha}-\sqrt{\beta})^2}{K} - \frac{\gamma(\log\alpha-\log\beta)}{2} - 1 > 0.
\end{align}
%According to $\alpha > \beta$ and $\sqrt{\alpha}-\sqrt{\beta} > \sqrt{2K}$, we have
%\begin{align*}
%\sqrt{\alpha+\beta} - \sqrt{2\beta} = \frac{\sqrt{\alpha}+\sqrt{\beta}}{\sqrt{\alpha+\beta}+\sqrt{2\beta}}(\sqrt{\alpha}-\sqrt{\beta}) > \frac{\sqrt{2}}{2}(\sqrt{\alpha}-\sqrt{\beta}) \ge \sqrt{K}. 
%\end{align*}
%This, together with $\sqrt{\alpha}-\sqrt{\beta} > \sqrt{2K}$, implies that there exist a constant $\eta > 0$, whose value only depends on $\alpha$, $\beta$ and $K$, such that
%\begin{align}\label{eq-2:lem-block-gap}
%c_3:=\frac{(\sqrt{\alpha+\beta}-\sqrt{2\beta})^2}{K} - \frac{\eta(\log(\alpha+\beta)-\log(2\beta))}{2} - 1 > 0.
%\end{align}
Since $\bA$ is generated according to the SBM in Definition \ref{SBM} with $p$ and $q$ satisfying \eqref{p-q}, one can verify that for all $i \in \mI_k$ with $\ell \neq k$,
\begin{align*}
 c_{ik} - c_{i\ell}   \overset{d}{=} \sum_{i=1}^m W_i - \sum_{i=1}^m Z_i,
\end{align*}
where $m=n/K$, $\{W_i\}_{i=1}^m$ are i.i.d.~$\mathbf{Bern}(\alpha\log n/n)$, and $\{Z_i\}_{i=1}^m$ are i.i.d.~$\mathbf{Bern}(\beta\log n/n)$ and independent of $\{W_i\}_{i=1}^m$. By Lemma \ref{lem:tail-Bino}, it holds that for any $\gamma \in \R$,
\begin{align*}
\P\left(  c_{ik} - c_{i\ell} \le \gamma \log n \right) \le n^{-\frac{(\sqrt{\alpha}-\sqrt{\beta})^2}{K}+\frac{\gamma(\log\alpha-\log\beta)}{K}}.
\end{align*}
This, together with the union bound and \eqref{eq-1:lem-block-gap}, implies
\begin{align}\label{eq-3:lem-block-gap}
\P\left(  c_{ik} - c_{i\ell} \ge \gamma \log n,\ \forall\ i \in \mI_k,\ 1\le k \neq \ell \le K \right) \ge 1 - Kn^{-c_2}.
\end{align}

\end{proof}

\subsection{Proof of Proposition \ref{prop:contra-PGM} }

\begin{proof} % and \eqref{rst-2:lem-block-gap} 
Suppose that \eqref{eq:spectral-norm-Delta} and \eqref{rst-1:lem-block-gap} hold, which happens with probability at least $1-n^{-3}-Kn^{-c_2}$ due to Lemma \ref{lem:spectral-norm-Delta}, Lemma \ref{lem:block-gap}, and the union bound. Let $\mI_k=\{i \in [n]: h_{ik}^*=1\}$ for all $k \in [K]$. This, together with \eqref{rst-1:lem-block-gap} and Lemma \ref{lem:closed-form-LP}, implies $\mT(\bA\bHs) = \{\bHs\}$. Besides, due to Lemma \ref{lem:LP-Q} and $\bV \in \mT(\bA\bH)$, we have $\bV\bQ^T \in \mT(\bA\bH\bQ^T)$ for some $\bQ \in \Pi_K$. According to these, \eqref{rst-1:lem-block-gap}, and Lemma \ref{lem:LP-cont}, we have for any $\bV\bQ^T \in \mT(\bA\bH\bQ^T)$, 
\begin{align*}
\|\bV - \bHs\bQ\|_F & =  \|\bV\bQ^T - \bHs\|_F \le \frac{2\|\bA\bH\bQ^T-\bA\bHs\|_F}{\gamma\log n} \\
& = \frac{2\|\bA\bH-\bA\bHs\bQ\|_F}{\gamma\log n}  \\
& \le \frac{8\varepsilon n(p-q)/\sqrt{K} + 2\|\Delta\|}{\gamma \log n} \|\bH-\bHs\bQ\|_F \\
% & \le \frac{10\varepsilon (\alpha-\beta)}{\delta\sqrt{K}} \|\bH-\bHs\|_F
& \le \left(\frac{8\varepsilon (\alpha-\beta)}{\gamma\sqrt{K}} + \frac{2c_1}{\gamma\sqrt{\log n}} \right) \|\bH-\bHs\bQ\|_F \\
& \le 4 \max\left\{ \frac{4\varepsilon (\alpha-\beta)}{\gamma\sqrt{K}}, \frac{c_1}{\gamma\sqrt{\log n}} \right\} \|\bH-\bHs\bQ\|_F,
\end{align*}
where the equalities are both0 because of $\bQ \in \Pi_K$, the second inequality is due to Lemma \ref{lem:contra}, and the third inequality follows from \eqref{p-q} and \eqref{eq:spectral-norm-Delta}. This implies the desired result in \eqref{rst:prop-PGM}. Since $\varepsilon < \gamma\sqrt{K}/(16(\alpha-\beta))$ and $n > \exp(16c_1^2/\gamma^2)$, then $\kappa$ defined in \eqref{conv-rate} satisfies $\kappa \in (0,1)$. Hence, the proof is completed. 
\end{proof}

\subsection{Proof of Lemma \ref{lem:one-step-conv}}
\begin{proof}
By letting $\bH^\prime = \bH\bQ^T$, it suffices to show $\mT(\bA\bH^\prime)=\{\bHs\}$ according to Lemma \ref{lem:LP-Q}. Suppose that \eqref{rst-1:lem-block-gap} holds, which happens with probability at least $1-Kn^{-c_2}$ according to Lemma \ref{lem:block-gap}, where $c_2>0$ is specified in \eqref{eq-1:lem-block-gap}. Let $\mI_k=\{i \in [n]: h_{ik}^* = 1\}$ and $\mJ_k=\{i \in [n]: h^\prime_{ik} = 1\}$ for all $k\in[K]$. Let $\mS_k=\mI_k \cap \mJ_k^c$ and $\mS_k^\prime = \mI_k^c \cap \mJ_k$ for all $k\in [K]$. 
According to \eqref{eq-7:lem-LP-cont}, we have $s_k := |\mS_k|=|\mS_k^\prime|$. Besides, one can verify 
\begin{align}\label{eq-1:lem-one-step}
\bh^\prime_k = \bh^*_k - \be_{\mS_k} + \be_{\mS_k^\prime},\ \forall\ k \in [K],
\end{align}
where  $\bh_k^*$ (resp. $\bh^\prime_k$) is the $k$-th column of $\bHs$ (resp.  $\bHp$ ) and $\be_{\mS_k}$ (resp. $\be_{\mS_k^\prime})$ is an $n$-dimensional vector with $(\be_{\mS_k})_i=1$ if $i\in \mS_k$ (resp. $\mS_k^\prime$) and $0$ otherwise. This, together with $\|\bHp-\bHs\|_F = \|\bH-\bHs\bQ\|_F < \sqrt{\gamma\log n}$, yields that for all $ k \in [K]$, 
\begin{align*}
2s_k = |\mS_k|+|\mS_k^\prime| = \|\be_{\mS_k} - \be_{\mS_k^\prime} \|^2 = \|\bh^\prime_k  - \bh^*_k\|^2 < \gamma\log n. 
\end{align*}
This implies 
\begin{align}\label{eq-2:lem-one-step}
|\mS_k|=|\mS_k^\prime| < \frac{\gamma}{2}\log n,\ \forall\ k\in [K].
\end{align} 
By letting $\bC^* = \bA\bHs$, $\bC = \bA\bHp$, and $\ba_i^T$ denote the $i$-th row of $\bA$, we have that for all $i \in \mI_k$ and $k\in [K]$,
\begin{align}\label{eq-3:lem-one-step}
c_{ik} = c_{ik}^* + \ba_i^T(\bh^\prime_k-\bh_k^*) =  c_{ik}^* + \ba_i^T(\be_{\mS_k^\prime}-\be_{\mS_k}) = c_{ik}^* + \sum_{j\in \mS_k^\prime}a_{ij} - \sum_{j\in \mS_k} a_{ij},
\end{align}
where the second equality is due to \eqref{eq-1:lem-one-step}. Now, for all  $i \in \mI_k$ with $1 \le k \neq \ell \le K$,  we have
\begin{align*}
c_{ik} - c_{i\ell} & = c_{ik}^* - c_{i\ell}^* +  \sum_{j\in \mS_k^\prime}a_{ij} - \sum_{j\in \mS_k} a_{ij} - \sum_{j\in \mS_\ell^\prime}a_{ij} + \sum_{j\in \mS_\ell} a_{ij} \\
& \ge \gamma\log n - |\mS_k| - |\mS_\ell^\prime|  \\
& >  0,
\end{align*}
where the equality is due to \eqref{eq-3:lem-one-step}, the first inequality uses \eqref{rst-1:lem-block-gap} and $a_{ij} \in \{0,1\}$, and the second inequality follows from \eqref{eq-2:lem-one-step}. This implies that for all $1\le k \neq \ell \le K$, $i \in \mI_k$, and $j\in \mI_\ell$,
\begin{align*}
c_{ik} - c_{i\ell} > 0 > c_{jk} - c_{j\ell}.
\end{align*}
According to Lemma \ref{lem:closed-form-LP}, we have that $\mT(\bA\bH^\prime)$ is a singleton and $ \mT(\bA\bH^\prime)=\{\bHs\}$. 
\end{proof}

\section{Proof of Theorem \ref{thm-2}}  

To simplify the notations in the proof, let
\begin{align}\label{r-phi}
r = \min\left\{\frac{1}{\sqrt{K}},\ \frac{\gamma\sqrt{K}}{16(\alpha-\beta)}\right\}\quad \text{and}\quad  \phi = \frac{c_1\sqrt{K}}{16(\alpha-\beta)}.
\end{align}

\begin{proof}
Suppose that the statements in Proposition \ref{prop:contra-PGM} and Lemma \ref{lem:one-step-conv} hold, which happens with probability at least $1-n^{-\Omega(1)}$ by the union bound. Let $\mI_k=\left\{i\in [n]:h^*_{ik}=1\right\}$ for all $k\in [K]$. Then, one can verify that $h_{ik}^*-h_{i\ell}^* = 1$ for all $i\in \mI_k$ and $k\neq \ell$, and thus $\bH^*\in \mT(\bH^*)$ by Lemma \ref{lem:closed-form-LP}. % , and $h_{ik}^*-h_{i\ell}^*+h_{j\ell}^*-h_{jk}^* \ge 1$ for all $i\in \mI_k$, $j\in \mI_{\ell^\prime}$, $k\neq \ell$, and $k \neq \ell^\prime$.
 This, together with  Lemma \ref{lem:LP-Q} and Lemma \ref{lem:LP-cont} with $\bH^1\in \mT(\bH^0)$ in Algorithm \ref{alg:PGD}, yields that $\bH^1 \in \H_{n,K}$ satisfies
\begin{align}\label{eq1:pf-thm-1}
\|\bH^1 - \bHs\bQ\|_F = \|\bH^1\bQ^T - \bHs\|_F  \le 2 \|\bH^0\bQ^T - \bHs\|_F = 2 \|\bH^0- \bHs\bQ\|_F. 
\end{align}
Let us divide our proof into two parts. We first show that for all $k \ge 2$, $\bH^k \in \H_{n,K}$ satisfies
\begin{align}\label{step1:pf-thm-1}
\|\bH^k - \bHs\bQ\|_F \le \frac{1}{2} \|\bH^{k-1}-\bHs\bQ\|_F\ \text{and}\ \|\bH^k-\bHs\bQ\|_F \le 2\theta\sqrt{n},
\end{align}
and compute the iteration number $N_1$ such that
\begin{align}\label{N1:pf-thm-1}
\|\bH^{N_1} - \bHs\bQ\|_F \le  2\phi\sqrt{\frac{n}{\log n}}.
\end{align}
Suppose that  $\bH^0\in \M_{n,K}$ satisfies \eqref{partial-recov}. According to \eqref{partial-recov} and \eqref{eq1:pf-thm-1}, we have
\begin{align}\label{eq3:pf-thm-1}
\bH^1 \in \H_{n,K}\ \text{and}\ \|\bH^1 - \bHs\bQ\|_F \le 2\theta\sqrt{n}.
\end{align}
This, together with $2\theta=r/2$, $\bH^2\in \mT(\bA\bH^1)$, and Proposition \ref{prop:contra-PGM}, yields 
\begin{align*}
\qquad \|\bH^2-\bHs\bQ\|_F 
% & \le 8\max\left\{ \frac{16\theta(\alpha-\beta)}{\eta\sqrt{K}},\frac{c_1}{\sqrt{\log n}} \right\}\|\bH^1\bQ^T-\bHs\|_F \\
\le 4\max\left\{ \frac{1}{8}, \frac{c_1}{\gamma\sqrt{\log n}} \right\}\|\bH^1-\bHs\bQ\|_F  =  \frac{1}{2}\|\bH^1-\bHs\bQ\|_F \le 2\theta\sqrt{n},
\end{align*}
where the first inequality follows from Proposition \ref{prop:contra-PGM} and \eqref{theta} and the equality is due to $n\ge \exp\left(64c_1^2/\gamma^2 \right)$. Thus, \eqref{step1:pf-thm-1} holds for $k=2$. By a simple inductive argument, we can show that \eqref{step1:pf-thm-1} holds for $k\ge 3$. As a result, \eqref{step1:pf-thm-1} can be established by a mathematical induction method. Let $N_1=\lceil 2\log\log n \rceil + 1$. It then follows from \eqref{step1:pf-thm-1} that
\begin{align*}
& \|\bH^{N_1} - \bHs\bQ\|_F \le \left(\frac{1}{2}\right)^{\lceil 2\log\log n \rceil}\|\bH^1 - \bHs\bQ\|_F  \le \left(\frac{1}{2}\right)^{ 2\log\log n }2\theta\sqrt{n} \le \left(\frac{1}{2}\right)^{\log\log n + 2\log\left(\frac{\theta}{\phi}\right)}2\theta\sqrt{n} \\
& \qquad \le \left(\frac{1}{2}\right)^{\frac{\log\log n + 2\log\left(\frac{\theta}{\phi}\right)}{2\log 2}}2\theta\sqrt{n} = 2\phi\sqrt{\frac{n}{\log n}},
\end{align*}
where the second inequality is due to \eqref{eq3:pf-thm-1}, the third inequality follows from $n \ge \exp\left(\gamma^2/c_1^2\right) \ge \exp\left(\theta^2/\phi^2\right)$, and the last inequality is due to $2\log 2 \ge 1$. Thus, \eqref{N1:pf-thm-1} holds for $N_1=\lceil 2\log\log n \rceil + 1$.

Next, we show that for all $k \ge 1$, $\bH^{N_1+k} \in \H_{n,K}$ satisfies $\|\bH^{N_1+k}-\bHs\bQ\|_F \le 2\phi\sqrt{n/\log n}$ and
\begin{align}\label{step2:pf-thm-1}
\|\bH^{N_1+k} - \bHs\bQ\|_F \le \frac{4c_1}{\gamma\sqrt{\log n}} \|\bH^{N_1+k-1}-\bHs\bQ\|_F,
\end{align}
and compute the iteration number $N_2$ such that 
\begin{align}\label{N2:pf-thm-1}
\|\bH^{N_2+N_1} - \bHs\bQ\|_F < \sqrt{\gamma\log n}. 
\end{align}
Since $n\ge \exp\left(\phi^2/\theta^2\right)$ and $2\theta=r/2$, it holds that $2\phi/\sqrt{\log n} \le 4\phi/\sqrt{\log n} \le r$. This, together with $\bH^{N_1}\in\H_{n,K}$, $\bH^{N_1+1}\in\mT(\bA\bH^{N_1})$, Proposition \ref{prop:contra-PGM}, and \eqref{phi}, yields
\begin{align*}
\|\bH^{N_1+1}-\bHs\bQ\|_F  \le 4\max\left\{ \frac{8\phi(\alpha-\beta)}{\gamma\sqrt{K\log n}},\frac{c_1}{\gamma\sqrt{\log n}} \right\}\|\bH^{N_1}-\bHs\bQ\|_F \le \frac{4c_1}{\gamma\sqrt{\log n}}\|\bH^{N_1}-\bHs\bQ\|_F.
\end{align*}
Then, \eqref{step2:pf-thm-1} holds for $k=1$. We can show that \eqref{step2:pf-thm-1} holds for $k\ge 2$ by a simple inductive argument. Thus, \eqref{step2:pf-thm-1} can be established by a mathematical induction method. Then, let $N_2=\left\lceil \frac{2\log n}{\log\log n} \right\rceil$. According to $n \ge \exp\left( 256c_1^4/\gamma^4\right)$ and $n > \exp\left(2\phi/\sqrt{\gamma}\right)$ in \eqref{n}, we have $\log\log n\ge 4\log(4c_1/\gamma)$ and $2\phi/\sqrt{\log n} < \sqrt{\gamma\log n}$. This, together with \eqref{step2:pf-thm-1}, yields
\begin{align*}
 \|\bH^{N_1+N_2} - \bHs\bQ\|_F  & \le   \left(\frac{4c_1}{\gamma\sqrt{\log n}} \right)^{\left\lceil \frac{2\log n}{\log\log n} \right\rceil} \|\bH^{N_1}-\bHs\bQ\|_F  \le 2\phi\sqrt{\frac{n}{\log n}} \left( \frac{4c_1}{\gamma\sqrt{\log n}} \right)^{\frac{2\log n}{\log\log n}} \\
& \le 2\phi\sqrt{\frac{n}{\log n}} \left( \frac{4c_1}{\gamma\sqrt{\log n}} \right)^{\frac{\log n}{\log\log n+2\log(\gamma/(4c_1))}}  = \frac{2\phi}{\sqrt{\log n}} < \sqrt{\gamma\log n}.
\end{align*}
Thus, \eqref{step2:pf-thm-1} holds for $N_2=\left\lceil \frac{2\log n}{\log\log n} \right\rceil$.

Once \eqref{N2:pf-thm-1} holds, we have $\bH^{N_1+N_2+1}=\bHs$  by Lemma \ref{lem:one-step-conv}. Then, the desired result is established. % This, together with Lemma \ref{lem:one-step-conv}, yields $\bH^{N_1+N_2+2}=\bHs$. According to the stopping criterion in Algorithm \ref{alg:PGD}, the desired result is established. 
\end{proof}

\end{document}